\documentclass[aop,preprint]{imsart}

\usepackage{amsfonts}
\usepackage{amsthm}
\usepackage[reqno]{amsmath}
\usepackage{amssymb}
\usepackage{mathtools}

\usepackage{enumitem}
\usepackage{bbm}
\usepackage{tikz}
\usetikzlibrary{arrows}
\usetikzlibrary[patterns]
\usetikzlibrary{decorations.pathreplacing}
\usetikzlibrary{shapes.misc}
\usetikzlibrary{hobby}

\usepackage[numeric,initials,nobysame]{amsrefs}
\usepackage{verbatim}
\usepackage{subcaption}
\usepackage{hyperref}

\arxiv{arXiv:1910.06782}

\startlocaldefs
\definecolor{ffqqqq}{rgb}{1,0,0}
\definecolor{qqffqq}{rgb}{0,1,0}
\definecolor{ffffff}{rgb}{1,1,1}
\colorlet{ColorGray}{gray!30}

\newtheorem{mainthm}{Theorem}
\newtheorem{thm}{Theorem}
\newtheorem{claim}[thm]{Claim}
\newtheorem{cor}[thm]{Corollary}
\newtheorem{lem}[thm]{Lemma}
\newtheorem{prop}[thm]{Proposition}
\newtheorem{conj}[thm]{Conjecture}

\theoremstyle{definition}
\newtheorem{defn}[thm]{Definition}
\newtheorem{rem}[thm]{Remark}

\newtheorem{obs}[thm]{Observation}
\newtheorem{ass}[thm]{Assumption}

\numberwithin{thm}{section}

\newcommand{\cA}{\ensuremath{\mathcal A}}

\newcommand{\cC}{\ensuremath{\mathcal C}}
\newcommand{\cD}{\ensuremath{\mathcal D}}
\newcommand{\cE}{\ensuremath{\mathcal E}}
\newcommand{\cF}{\ensuremath{\mathcal F}}
\newcommand{\cG}{\ensuremath{\mathcal G}}
\newcommand{\cH}{\ensuremath{\mathcal H}}

\newcommand{\cP}{\ensuremath{\mathcal P}}

\newcommand{\cS}{\ensuremath{\mathcal S}}

\newcommand{\cU}{\ensuremath{\mathcal U}}

\newcommand{\bbE}{{\ensuremath{\mathbb E}} }

\newcommand{\bbH}{{\ensuremath{\mathbb H}} }

\newcommand{\bbN}{{\ensuremath{\mathbb N}} }

\newcommand{\bbP}{{\ensuremath{\mathbb P}} }
\newcommand{\bbQ}{{\ensuremath{\mathbb Q}} }
\newcommand{\bbR}{{\ensuremath{\mathbb R}} }

\newcommand{\bbX}{{\ensuremath{\mathbb X}} }

\newcommand{\bbZ}{{\ensuremath{\mathbb Z}} }

\let\a=\alpha    \let\d=\delta  
 \let\g=\gamma     \let\k=\kappa  \let\l=\lambda
\let\m=\mu   \let\n=\nu   \let\o=\omega      
\let\r=\rho  \let\s=\sigma \let\t=\tau   
 \let\x=\xi 
\let\D=\Delta   \let\G=\Gamma  \let\L=\Lambda 
\let\O=\Omega

\newcommand{\tc}{\thinspace |\thinspace}
\newcommand{\var}{\operatorname{Var}}
\newcommand{\<}{\langle}
\renewcommand{\>}{\rangle}

\newcommand{\1}{{\ensuremath{\mathbbm{1}}} }
\newcommand{\Aint}{{\ensuremath{A^{\text{int}}}} }
\newcommand{\ur}{{\ensuremath{\underbar r}} }
\newcommand{\oB}{{\ensuremath{B^\circ} }}
\newcommand{\ds}{{\mathrm{d}s}\ }

\renewcommand{\geq}{\geqslant}
\renewcommand{\le}{\leqslant}
\renewcommand{\ge}{\geqslant}
\renewcommand{\to}{\rightarrow}

\newcommand{\hS}{\ensuremath{\widehat\cS}}

\newcommand{\SG}{\ensuremath{\mathcal {SG}}}

\newcommand{\HA}{\ensuremath{\mathcal {HA}}}
\newcommand{\SB}{\ensuremath{\mathcal {SB}}}
\newcommand{\ST}{\ensuremath{\mathcal {ST}}}

\tikzset{cross/.style={cross out, draw=black, minimum size=2*(#1-\pgflinewidth), inner sep=0pt, outer sep=0pt},
cross/.default={1pt}}

\pgfdeclarepatternformonly[\LineSpace,\tikz@pattern@color]{my north east lines}{\pgfqpoint{-1pt}{-1pt}}{\pgfqpoint{\LineSpace}{\LineSpace}}{\pgfqpoint{\LineSpace}{\LineSpace}}%
{
    \pgfsetcolor{\tikz@pattern@color}
    \pgfsetlinewidth{0.4pt}
    \pgfpathmoveto{\pgfqpoint{0pt}{0pt}}
    \pgfpathlineto{\pgfqpoint{\LineSpace + 0.1pt}{\LineSpace + 0.1pt}}
    \pgfusepath{stroke}
}
\makeatother

\newdimen\LineSpace
\tikzset{
    line space/.code={\LineSpace=#1},
    line space=10pt
}

\endlocaldefs

\begin{document}

\begin{frontmatter}

\title{Universality for critical KCM:\\finite number of stable directions\thanksref{t1}}
\runtitle{Universality for critical KCM}
\thankstext{T1}{The authors were supported by ERC Starting Grant 680275 ``MALIG''. The second author was supported by PRIN 20155PAWZB ``Large Scale Random Structures''. The third author was supported by ANR-15-CE40-0020-01.}

\begin{aug}
\author[A]{\fnms{Ivailo} \snm{Hartarsky}\ead[label=e1,mark]{hartarsky@ceremade.dauphine.fr}},
\author[B]{\fnms{Fabio} \snm{Martinelli}\ead[label=e2,mark]{martin@mat.uniroma3.it}}
\and
\author[A]{\fnms{Cristina} \snm{Toninelli}\ead[label=e3,mark]{toninelli@ceremade.dauphine.fr}}
\address[A]{CEREMADE, CNRS, UMR 7534, Universit\'e Paris-Dauphine, PSL University, Place du Mar\'echal de Lattre de Tassigny, 75775 Paris Cedex 16, France, 
\printead{e1,e3}}

\address[B]{Dipartimento di Matematica e Fisica, Universit\`a Roma Tre, Largo S.L. Murialdo, 00146, Roma, Italy,
\printead{e2}}
\runauthor{I. Hartarsky, F. Martinelli, C. Toninelli}
\end{aug}

\begin{abstract}
In this paper we consider kinetically constrained models (KCM) on $\bbZ^2$ with general update families $\cU$. 
For $\cU$ belonging to the so-called ``critical class'' our focus is on the divergence of the infection time of the origin for the equilibrium process as the density of the facilitating sites vanishes. In a recent paper \cite{Hartarsky20} Mar\^ech\'e and two of the present authors proved that if $\cU$ has an infinite number of ``stable directions,'' then on a doubly logarithmic scale the above divergence is twice the one in the corresponding $\cU$-bootstrap percolation. 

Here we prove instead that, contrary to previous conjectures \cite{Martinelli19a}, in the complementary case the two divergences are the same. In particular, we establish  the full universality partition for critical $\cU$. The main novel contribution is the identification of the leading mechanism governing the motion of infected critical droplets. It consists of a peculiar hierarchical combination of mesoscopic East-like motions. 
\end{abstract}



\end{frontmatter}

\section{Introduction}
\subsection{Kinetically constrained models}
We directly define the models of
interest and refer the reader to the companion paper \cite{Hartarsky20} for more background. Let $\cU$ be a finite collection of finite subsets of
$\bbZ^2\setminus \{0\}$ called \emph{update rules} and consider the following 
interacting particle systems on $\O=\{0,1\}^{\bbZ^2}$ parametrised by
$\cU$ and $q\in (0,1)$. We say that a site $x\in\bbZ^2$ is \emph{infected} if $\o(x)=0$ and \emph{healthy} otherwise. The dynamics is as follows.  On each site $x\in\bbZ^2$ we are given an independent Poisson clock of parameter one and at each arrival time $t:=t_{x,k}$ of the clock the process attempts to update the current state $\o_x(t)$ according to the following rule. If the configuration $\omega(t)$ is such that there exists $ U\in\cU$ such that $\forall y\in U,\omega_{x+y}(t)=0,$ then $\omega_x(t)$ is resampled from the Bernoulli$(1-q)$-measure $\mu_q(1)=1-q,\mu_q(0)=q$. In this case we say that a~\emph{legal update} occurs at time $t$ at site $x$. Otherwise the attempted update is rejected.

Using the fact that all rings $(t_{x,k})_{x\in \bbZ^2,k\in \bbN}$ of the Poisson clocks are different a.s.\ and that all updating rules are finite sets, it is easy to check that the above process is well defined \cite{Liggett05}. Moreover, since the update rules do not contain the origin, the process is reversible w.r.t.\ the product measure $\mu:= \mu_q^{\bbZ^2}$.

We will refer to the above process as the \emph{kinetically constrained spin model} with update family $\cU$, for short
$\cU$-KCM or just KCM if $\cU$ is clear from the context. KCM have
been introduced several years ago in the physics literature
(but only for certain specific choices of $\cU$) in order to
reproduce in simple and fundamental interacting particle systems
some of the main features of the so-called \emph{glassy dynamics}, i.e.\ the
dynamics of a supercooled liquid near the \emph{glass
  transition} \cites{Fredrickson84,Jackle91,Ritort03,Garrahan11,Berthier11}.

\subsection{$\cU$-Bootstrap percolation}
The $\cU$-KCM  can also be seen as the non-monotone stochastic counterpart of the \emph{$\cU$-bootstrap percolation}, a  monotone cellular automaton on $\O$ (see e.g.\ \cites{Bollobas14,Morris17}). In the latter process one says that $x\in\bbZ^2$ is \emph{infected} for $\o\in \O$  if $\o_x=0$ and \emph{healthy}
otherwise and the relevant time evolution concerns the set of
infected sites $A_t$ at integer times $t$. Given $A_t,$ the set
$A_{t+1}$ is constructed by adding to $A_t$ any healthy site $x$ for
which there exists $U\in \cU$ such that $U+x\subset A_t$:
\[
A_{t+1}=A_t\cup\{x\in\bbZ^2,\exists U\in \cU,
x+U\subset A_t\}.
\]
The only
randomness in the process occurs at time $t=0$ by taking the
initial infection as the random set $A_\o=\{x\in \bbZ^2:\o_x=0\}$ with
$\o\sim \mu$. If $A_{t=0}=A$ then the \emph{closure} of $A$ under the $\cU$-bootstrap percolation is the set $[A]_\cU := \bigcup_{t \ge 0} A_t$.

For both processes the main focus has been on the
typical value (e.g.\ in mean, median, or w.h.p.) of the infection time of the origin defined as
\[
  \t_0=\inf\{t: \o_0(t)=0\}.
\]
Notice that for $\cU$-bootstrap percolation $\t_0\in \bbN\cup \{+\infty\}$ and
it only depends on the
initial set of infection $A_\o$. On the other hand, for the KCM $\t_0\in [0,+\infty]$ and it depends on the initial state $\o(0)$, on the occurrences of the Poisson clocks at the vertices of $\bbZ^2$ and on the resampling values of the legal updates.
In order to present our main result on $\t_0$ we need some further notation (see \cite{Bollobas15}).

Let $\|\cdot\|$ and $\<\cdot,\cdot\>$ denote the Euclidean norm and scalar product on $\bbR^2$ respectively. For each unit vector $u \in S^1=\{x\in\bbR^2,\|x\|=1\}\sim\bbR/2\pi\bbZ$, let 
$\bbH_u := \{x \in \bbR^2 : \langle x,u \rangle < 0 \}$
denote the half-plane whose boundary is perpendicular to $u$. 

\begin{defn}[Stable directions]
  \label{def:stable}
The set of \emph{stable directions} of $\cU$ is
\[\cS = \cS(\cU) = \left\{ u \in S^1: [\bbH_u\cap \bbZ^2]_\cU =
\bbH_u \cap \bbZ^2\right\}.\]
\end{defn}
Using the above definition the update family $\cU$ was classified in
\cite{Bollobas15} as:
\begin{itemize}
\item \emph{supercritical} if there exists an open semicircle in $S^1$ that is disjoint from $\cS$,
\item \emph{critical} if there exists a semicircle in $S^1$ that has finite intersection with $\cS$, and if every open semicircle in $S^1$ has non-empty intersection with $\cS$, \vspace{0.1cm}
\item \emph{subcritical} if every semicircle in $S^1$ has infinite intersection with $\cS$. 
\end{itemize}
It is known from~\cites{Bollobas15,Balister16} that for all $q\in (0,1)$ the $\cU$-bootstrap percolation $\tau_0$ is a.s.\ finite iff $\cU$ is supercritical or critical. The next definition quantifies the difficulty of propagation of infection in a stable direction for bootstrap percolation.
\begin{defn}[Definition~1.2 of~\cite{Bollobas14}]
\label{def:stable:alpha}
Let $\cU$ be an update family and $u\in S^1$ be a direction. The \emph{difficulty} of $u$, $\alpha(u)$, is defined as follows.
\begin{itemize}
    \item If $u$ is unstable, then $\alpha(u)=0$.
    \item If $u$ is an isolated (in the topological sense) stable direction, then
    \begin{equation*}
    \alpha(u)=\min\{n\in\bbN \colon \exists Z\subset\bbZ^2,|Z|=n,|[\bbZ^2\cap(\bbH_u\cup Z)]\setminus\bbH_u|=\infty\},
    \end{equation*}
    i.e.\ the minimal number of infections allowing $\bbH_u$ to grow infinitely.
    \item Otherwise, $\alpha(u)=\infty$.
\end{itemize}
The \emph{difficulty} of $\cU$ is 
\begin{equation}
\label{eq:def:alpha}
\alpha(\cU)=\inf_{C\in\cC}\sup_{u\in C}\alpha(u),
\end{equation}
where $\cC$ is the set of open semicircles of $S^1$. 
\end{defn}
\begin{rem}
\label{rem:finite difficulty}It was proved in \cite {Bollobas15}*{Lemma 5.2} (see also~\cite{Bollobas14}*{Lemma 2.8}) that $1\le \alpha(u)<\infty$ if and only if $u$ is an isolated stable direction. Moreover, it is easy to prove \cite{Bollobas14}*{Lemma 6.6} that if $\alpha(u)<+\infty$ then there exists a set $Z$ of cardinality $\alpha(u)$ such that $|[\bbZ^2\cap(\bbH_u\cup Z]\cap \ell(u)|=+\infty,$ where $\ell_u=\{x\in \bbR^2: \< x,u\>=0\}$. 
\end{rem}
It follows from the main result of \cite{Bollobas14} that for any critical $\cU$ with difficulty $\a$ the $\cU$-bootstrap percolation infection time $\t_0$ w.h.p.\ satisfies
\begin{equation}
  \label{eq:1.1}
  \lim_{q\to 0}\frac{\log\log (\t_0)}{\log(1/q)}=\a.
\end{equation}

\subsection{Main results}
Our main result is that \eqref{eq:1.1} holds also for the $\cU$-KCM if $\cS$ is finite. The core of the proof is based on the discovery of a new and efficient relaxation mechanism completely different from the one occurring in bootstrap percolation. While for the latter the dominant mechanism to grow infection is a linear expansion from some rare large groups of infected sites (critical droplets), for KCM we find that these droplets, in order to move around in an efficient way to infect the origin, perform a complex hierarchical motion (see Section \ref{sec:sketch} for an heuristic detailed description). The above motion is a novel type of relaxation mechanism with respect to all those considered so far in the KCM literature. In particular, it is different from the random walk like motion that captures the dominant behavior for 2-neighbour model \cite{Martinelli19}, and it is also different from the purely East-like motion used to establish the scaling for models with an infinite number of stable directions \cites{Martinelli19a,Mareche20Duarte,Hartarsky20}. Indeed, based on the wrong intuition that the two former mechanisms were essentially the only two possible efficient ways to move critical droplets around, a conjecture was put forward in \cite{Martinelli19a}*{Conjecture 3}, which is disproved by our result, Theorem \ref{th:main} below.

Write $\bbE_\mu(\t_0)$ for the expectation of the infection time for the $\cU$-KCM with initial law $\mu$ (i.e.\ for the stationary process).
\begin{mainthm}
\label{th:main}
Let $\cU$ be a critical update family with \emph{finite} stable set $\cS$ and difficulty $\a$. Then 
\begin{equation}
  \label{eq:1.2}
  \bbE_\mu(\tau_0)=e^{O(\log(1/q)^3/q^\alpha)}.
\end{equation}
Moreover, 
\begin{equation}
  \label{eq:1.3}
  \lim_{q\to 0}\frac{\log\log (\bbE_\mu(\t_0))}{\log(1/q)}=\a.
\end{equation}
\end{mainthm}
The second statement \eqref{eq:1.3} follows immediately from \eqref{eq:1.2} together
with  \cite{Martinelli19}*{Lemma 4.3} and \eqref{eq:1.1}.
Theorem \ref{th:main} together with \cite{Hartarsky20}*{Theorem 2.8} and \cite{Martinelli19a}*{Theorem 2(a)} corrects \cite{Martinelli19a}*{Conjecture 3} and gives the following full universality picture for $\cU$-KCM with critical $\cU$.
\begin{mainthm}
  \label{th:main2}
Let $\cU$ be a critical update family. Then  
\begin{equation}
\label{eq:1.4}  \lim_{q\to 0}\frac{\log\log (\bbE_\mu(\t_0))}{\log(1/q)}=
  \begin{cases}
    \a & \text{if $|\cS|<+\infty$,}\\
    2\a & \text{otherwise.}
  \end{cases}
\end{equation}
\end{mainthm}
Some key partial results in the direction of proving Theorem \ref{th:main2} were established in \cites{Martinelli19, Martinelli19a}. In
particular, in \cite{Martinelli19a} the scaling \eqref{eq:1.3} was proved for
any update family
$\cU$ with $\max_{u\in S^1}\a(u)=\a$ while \cite{Mareche20Duarte} proved \eqref{eq:1.4} with a higher degree of precision for a specific model with $|\cS|=\infty$, the Duarte model. We refer the interested reader to the
introduction of \cite{Hartarsky20} for a detailed account of the history
leading to \eqref{eq:1.4}. For supercritical update families the correct scaling was determined in \cite{Martinelli19a}*{Theorem 1} and \cite{Mareche20Duarte}*{Corollary 4.3}.

\subsection{Organisation of the paper}
We start by providing a heuristic explanation of the relaxation mechanism underlying our main result in Section \ref{sec:sketch}. In Section \ref{sec:prelim} we fix some notation and gather some
preliminary tools from bootstrap percolation that are by now well
established in the literature. We will not dwell on the technical
aspects of the definitions and invite the reader to refer to Section
4.3 of~\cite{Martinelli19a}, which we follow closely, for more details. For
reader's convenience we have collected in Section \ref{sec:tools}
three useful technical lemmas on certain one-dimensional kinetically constrained Markov processes. Although the proof of these lemma can be found or derived from the existing literature on KCM, we have added it in the \hyperref[app]{Appendix} for completeness. Section \ref{sec:core} contains the main new technical Poincar\'e inequality, while Theorem \ref{th:main} is proved in Section \ref{sec:Proof}. Finally, in Section \ref{sec:open pb} we discuss some natural open problems raised by the present work.

\section{Some heuristics behind Theorem \ref{th:main}}
\label{sec:sketch}
For a high-level and accessible introduction to
the main general ideas and techniques involved in bounding from above
$\bbE_\mu(\t_0)$ we refer to
\cite{Martinelli19a}*{Section 2.4}. There, in particular, it was stressed that while the necessary intuition is developed using dynamical considerations (e.g.\ by guessing some efficient mechanism to create/heal infection inside the system), the actual mathematical tools are mostly analytic and based on suitable (and, unfortunately, sometimes very technical) Poincar\'e inequalities. This paper makes no exception.

In order to go beyond the
results of \cite{Martinelli19a} and get the sharp scaling of Theorem
\ref{th:main} in the case of a \emph{finite} set of stable directions,
the following new key input is needed. 

For simplicity imagine that $\cU$ has only four stable
directions coinciding with the four natural directions of $\bbZ^2$. For a generic model with $|\cS|<\infty$ the mechanism is the same, the only difference being that in general `droplets' have a more complex geometry. Assume further that $\a(\vec e_1)=1$ and $\a(-\vec e_1)=\a(\pm \vec e_2)=2$ (see Figure \ref{fig:example}).
\begin{figure}
\begin{subfigure}{0.45\textwidth}
\begin{center}
\begin{tikzpicture}[line cap=round,line join=round,x=1.0cm,y=1.0cm, scale=0.7]
\draw [dash pattern=on 3pt off 3pt, xstep=1cm,ystep=1cm] (-2,-1) grid (0,0);
\clip(-2.5,-1.5) rectangle (0.5,0.5);
\draw (5pt,5pt) node {$0$};
\fill (-1,0) circle (2.5pt);
\fill (-2,0) circle (2.5pt);
\fill (0,-1) circle (2.5pt);
\end{tikzpicture}
~
\begin{tikzpicture}[line cap=round,line join=round,x=1.0cm,y=1.0cm, scale=0.7]
\draw [dash pattern=on 3pt off 3pt, xstep=1cm,ystep=1cm] (-2,0) grid (0,1);
\clip(-2.5,-0.5) rectangle (0.5,1.5);
\draw (5pt,-5pt) node {$0$};
\fill (-1,0) circle (2.5pt);
\fill (-2,0) circle (2.5pt);
\fill (0,1) circle (2.5pt);
\end{tikzpicture}

\begin{tikzpicture}[line cap=round,line join=round,x=1.0cm,y=1.0cm, , scale=0.7]
\draw [dash pattern=on 3pt off 3pt, xstep=1cm,ystep=1cm] (0,-2) grid (2,0);
\clip(-0.5,-2.5) rectangle (2.5,0.5);
\draw (-5pt,5pt) node {$0$};
\fill (1,0) circle (2.5pt);
\fill (2,0) circle (2.5pt);
\fill (0,-1) circle (2.5pt);
\fill (0,-2) circle (2.5pt);
\end{tikzpicture}
~
\begin{tikzpicture}[line cap=round,line join=round,x=1.0cm,y=1.0cm,, scale=0.7]
\draw [dash pattern=on 3pt off 3pt, xstep=1cm,ystep=1cm] (0,0) grid (2,2);
\clip(-0.5,-0.5) rectangle (2.5,2.5);
\draw (-5pt,-5pt) node {$0$};
\fill (1,0) circle (2.5pt);
\fill (2,0) circle (2.5pt);
\fill (0,1) circle (2.5pt);
\fill (0,2) circle (2.5pt);
\end{tikzpicture}
\end{center}
\caption{The four update rules belonging to $\cU$. }
\end{subfigure}
\quad
\begin{subfigure}{0.45\textwidth}
\begin{center}
\begin{tikzpicture}[line cap=round,line join=round,x=2.0cm,y=2.0cm, scale=0.7]
\clip(-1.25,-1.25) rectangle (1.25,1.25);
\draw(0,0) circle (2cm);
\draw (0,0)-- (1,0);
\draw (0,1)-- (0,0);
\draw (0,0)-- (-1,0);
\draw (0,0)-- (0,-1);
\fill [color=ffqqqq] (1,0) circle (2.5pt);
\draw (1.15,0) node {$1$};
\fill [color=ffqqqq] (0,1) circle (2.5pt);
\draw (0,1.15) node {$2$};
\fill [color=ffqqqq] (-1,0) circle (2.5pt);
\draw (-1.15,0) node {$2$};
\fill [color=ffqqqq] (0,-1) circle (2.5pt);
\draw (0,-1.15) node {$2$};
\end{tikzpicture}
\end{center}
\caption{The difficulties of the four stable directions of the model.}
\end{subfigure}
\caption{
  \label{fig:example}
  An example of an update family $\cU$ with finite number of stable directions examined in Section \ref{sec:sketch}.}
\end{figure}
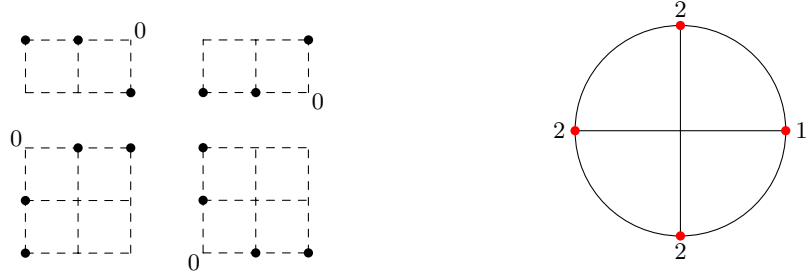
Consider now a \emph{critical droplet}, i.e.\ a square frame $D$, centered at
the origin, of side length $\approx
C\log(1/q)/q$, $C\gg 1,$ and $O(1)$-thickness, and suppose that $D$ is infected. Then,  w.h.p.\ (w.r.t.\ $\mu$) there will be extra infected sites next to $D$ in the $\vec e_1$-direction allowing $D$ to infect $D+\vec e_1$. However, it will be extremely unlikely
to find a pair of infected sites near each other and next to the other
three sides of $D$ because of the choice of the side length of $D$. We
conclude that w.h.p.\ it is easy for $D$ to advance forward in the $\vec
e_1$-direction but not in the other
directions. Moreover, as explained in detail in \cite{Martinelli19a}*{Section 2.4}, an
efficient way to effectively realize the motion in
the $\vec
e_1$-direction is via a generalised East path. In its essence the
latter can be described  by the following game. At every integer time a token is added or removed (if already present) at some integer point
according to the following rules:
\begin{itemize}
\item each integer can accomodate at most one
token;
\item a token can be freely added or removed at $1$;
\item for any $j\ge 2$ the operation of adding/removing a
  token at $j$ is allowed iff there is already a token at $j-1$.     
\end{itemize}
Given $n\in \bbN$, by an efficient path reaching distance $n$ we mean a way of 
adding tokens to the original empty configuration to finally place one
at $n$ which uses a minimal number of tokens. A combinatorial result (see \cite{Chung01})
says that the optimal number grows like $\log_2(n)$.

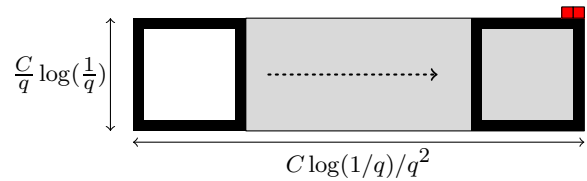
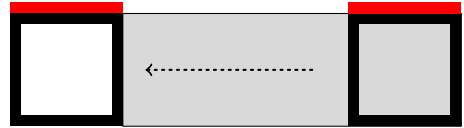
\begin{figure}
  \begin{subfigure}{1.0\linewidth}
  \centering
\begin{tikzpicture}[line cap=round,line join=round,x=0.25cm,y=0.25cm,scale=0.6]
\clip(-12,-5) rectangle (52,11);
\draw [fill=ColorGray] (10,0) rectangle (40,10);
  \fill[fill=black,fill opacity=1.0] (0,0) -- (10,0) -- (10,10) -- (0,10) -- cycle;
\fill[line width=0pt,color=ffffff,fill=ffffff,fill opacity=1.0] (1,9) -- (9,9) -- (9,1) -- (1,1) -- cycle;
\fill[fill=black] (40,10) -- (30,10) -- (30,0) -- (40,0) -- cycle;
\fill[line width=0pt,color=ffffff,fill=ColorGray,fill opacity=1.0] (31,1) -- (39,1) -- (39,9) -- (31,9) -- cycle;
\draw[fill=ffqqqq,fill opacity=1.0] (40,11) -- (40,10) -- (39,10) -- (39,11) -- cycle;
\draw[fill=ffqqqq,fill opacity=1.0] (38,11) -- (38,10) -- (39,10) -- (39,11) -- cycle;
\draw [thick,->, dotted] (12,5)--(27,5);
\draw [<->] (-2,0) -- (-2,10);
\draw [<->] (0,-1) -- (40,-1);
\draw (20,-0.8) node[anchor=north] {$C\log(1/q)/q^2$};
\draw (-1.5,5) node[anchor=east] {$\frac{C}{q}\log(\frac 1q)$};
\end{tikzpicture}
\caption{The infected droplet (black frame with width $O(1)$) progressively moves to the right in an East-like way using the extra infected sites present  w.h.p.\ in each column of the gray rectangle. This progression stops when reaching the first infected horizontal pair of sites at the correct height (red pair). }
\end{subfigure}
\begin{subfigure}{1.0\linewidth}
\centering
\begin{tikzpicture}[line cap=round,line join=round,x=0.25cm,y=0.25cm,scale=0.6]
  \clip(-12,-0.5) rectangle (52,11);
\draw [fill=ColorGray] (10,0) rectangle (40,10);
\fill[color=ffqqqq,fill=ffqqqq,fill opacity=1.0] (0,11) -- (0,10) -- (10,10) -- (10,11) -- cycle;
\fill[fill=black,fill opacity=1.0] (0,0) -- (10,0) -- (10,10) -- (0,10) -- cycle;
\fill[line width=0pt,color=ffffff,fill=ffffff,fill opacity=1.0] (1,9) -- (9,9) -- (9,1) -- (1,1) -- cycle;
\fill[fill=black] (40,10) -- (30,10) -- (30,0) -- (40,0) -- cycle;
\fill[line width=0pt,color=ffffff,fill=ColorGray,fill opacity=1.0] (31,1) -- (39,1) -- (39,9) -- (31,9) -- cycle;
\fill[color=ffqqqq,fill=ffqqqq,fill opacity=1.0] (40,11) -- (40,10) -- (30,10) -- (30,11) -- cycle;

\draw [thick,<-, dotted] (12,5)--(27,5);

\end{tikzpicture}
\caption{The infected droplet on the right grows into an $e_2$-extended droplet thanks to the infected pair of sites. The movement is then reverted, progressively retracting the extended droplet in an East way until reaching the original position.}
\end{subfigure}
 \caption{
\label{fig:mechanism:1}
The mechanism for the droplet to grow in the $\vec e_2$-direction.}
\end{figure}
The main new idea now is that, while w.h.p.\ the droplet $D$ will not
find a pair of infected sites (which are necessary to grow an extra
layer of infection in the $\vec e_2$-direction) \emph{next} to
e.g.\ its top side, w.h.p.\ it will find it at the right height within distance $C\log(1/q)/q^2$ in the $\vec
e_1$-direction (see Figure \ref{fig:mechanism:1}). Hence, a possible efficient way for $D$ to move one step in the $\vec e_2$-direction is to:
\begin{enumerate}[label=(\alph*),ref=\alph*]
\item 
\label{step:A}travel in the $\vec e_1$-direction  in a East-like way  until finding the necessary pair of infected sites within distance $C\log(1/q)/q^2$ from the origin;
\item grow there an extra layer in the $\vec e_2$-direction and retrace back to its original position while keeping the acquired extra layer
of infection.              
\end{enumerate}
A similar mechanism applies to the $-\vec e_2$-direction. Slightly more involved is the way in which $D$ can advance in the $-\vec e_1$-direction. In this case the extra infected pair needs to be found within distance $C\log(1/q)/q^2$ from the origin in the \emph{vertical} direction (see Figure \ref{fig:mechanism:2}). In order to reach it, $D$ performs an East-like movement upwards, each of whose steps is itself realised by the back-and-forth East motion in the $e_1$ direction described above. 
\begin{figure}
 \centering
\begin{tikzpicture}[line cap=round,line join=round,x=0.25cm,y=0.25cm,scale=0.6]
  \begin{scope}[shift={(-10,15)}]
  \fill[fill=black,fill opacity=1.0] (0,0) -- (10,0) -- (10,10) -- (0,10) -- cycle;
\fill[line width=0pt,color=ffffff,fill=ffffff,fill opacity=1.0] (1,9) -- (9,9) -- (9,1) -- (1,1) -- cycle;
\fill[fill=black] (40,10) -- (30,10) -- (30,0) -- (40,0) -- cycle;
\fill[line width=0pt,color=ffffff,fill=ffffff,fill opacity=1.0] (31,1) -- (39,1) -- (39,9) -- (31,9) -- cycle;
\draw[fill=ffqqqq,fill opacity=1.0] (40,11) -- (40,10) -- (39,10) -- (39,11) -- cycle;
\draw[fill=ffqqqq,fill opacity=1.0] (38,11) -- (38,10) -- (39,10) -- (39,11) -- cycle;
\draw [thick,->, dotted] (12,3)--(27,3);
\draw [thick,<-, dotted] (12,6)--(27,6); 

\fill[color=ffqqqq,fill=ffqqqq,fill opacity=1.0] (0,11) -- (0,10) -- (10,10) -- (10,11) -- cycle;

 \end{scope}
  
  \begin{scope}[rotate=90]
  \fill[fill=black,fill opacity=1.0] (0,0) -- (10,0) -- (10,10) -- (0,10) -- cycle;
\fill[color=ffqqqq,fill=ffqqqq,fill opacity=1.0] (0,11) -- (0,10) -- (10,10) -- (10,11) -- cycle;
\fill[line width=0pt,color=ffffff,fill=ffffff,fill opacity=1.0] (1,9) -- (9,9) -- (9,1) -- (1,1) -- cycle;
\fill[fill=black] (40,10) -- (30,10) -- (30,0) -- (40,0) -- cycle;
\fill[line width=0pt,color=ffffff,fill=ffffff,fill opacity=1.0] (31,1) -- (39,1) -- (39,9) -- (31,9) -- cycle;
\draw[fill=ffqqqq,fill opacity=1.0] (40,11) -- (40,10) -- (39,10) -- (39,11) -- cycle;
\draw[fill=ffqqqq,fill opacity=1.0] (38,11) -- (38,10) -- (39,10) -- (39,11) -- cycle;
\draw [thick, ->,dotted] (11,5)--(13,5);
\draw [thick, ->,dotted] (26.5,5)--(28.5,5);
\draw [thick, <-,dotted] (11,7)--(13,7);
\draw [thick, <-,dotted] (26.5,7)--(28.5,7);
\end{scope}
\draw [<->] (-15,0) -- (-15,40);
\draw (-30,20) node[anchor=north west] {$\frac{C}{q^2}\log(1/q)$};
\end{tikzpicture}
\caption{The mechanism for the droplet growth in the $-\vec
  e_1$-direction. The droplet moves in an East way in the $\vec
  e_2$-direction by
  making long excursions in the $\vec e_1$-direction as in 
  Figure~\ref{fig:mechanism:1}.
}
\label{fig:mechanism:2}
\end{figure}
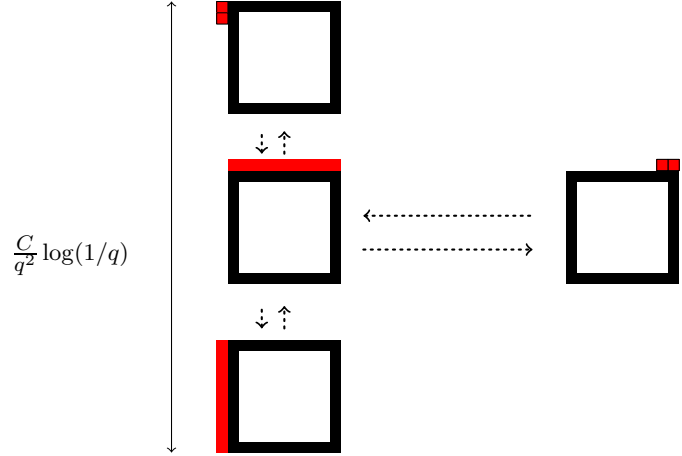

\noindent
Using the result for the typical time scales of the generalised East
process (see \cite{Martinelli19a}*{formula (3.5)}) it is easy to see that the
typical excursion of $D$ for a distance $\ell\equiv C\log(1/q)/q^2$ in the $\vec e_1$-direction requires a time lag
\[
\D t=  q^{-|D|O(\log(\ell))}= e^{O(\log^3(1/q))/q}.
\]
This time scale also bounds from above the time scale necessary to advance by one step in the ``hard'' directions $-\vec e_1,\pm \vec e_2$. 

In conclusion, by making a ``quasi-local''  (\emph{i.e} on a length
scale $\ell$) East-like motion in the easy direction $\vec e_1$, the
infected critical droplet $D$ can actually perform a sort of random
walk in which each step requires a time $\D t$. The
result of Theorem \ref{th:main} becomes now plausible provided that
one proves that anomalous regions of missing helping infected sites do not really constitute a serious obstacle. 

The above dynamic heuristics can be turned into a rigorous argument
using canonical
paths. However, a much neater approach is to prove a 
Poincar\'e inequality for the $\cU$-KCM restricted to a suitable
\emph{finite} domain of $\bbZ^2$ (see Theorem \ref{thm:key step}).   
More precisely, in the toy example discussed above the inequality that
we establish is as follows.

Let $V= B\cup T_0\cup T_1^-\cup T_1^+$ where
$B,T_0,T_1^\pm$  are as in Figure \ref{fig:mechanism:3}. The set $V$ is an example of a more general geometric construction developed in Section \ref{sec:core} and denoted \emph{snail with base $B$ and trapezoids $T_0,T_1^\pm$}. The ratio of the sides of the 
rectangle $B$ is $\Theta(q)$ while for the other rectangles it is $\Theta(1)$. 
\begin{figure}
 \centering
\begin{tikzpicture}[line cap=round,line join=round,x=0.25cm,y=0.25cm,scale=0.8]
\draw [thick] (0,0) rectangle (30,5);
\draw [thick] (0,5.3) rectangle (30,20.3);
\draw [thick] (-0.3,0) rectangle (-15.3,20.3);
\draw [thick] (30.3,0) rectangle (45.3,20.3);
\fill[fill=black,fill opacity=1.0] (0,0) -- (5,0) -- (5,5) -- (0,5) -- cycle;
\fill[line width=0pt,color=ffffff,fill=ffffff,fill opacity=1.0] (0.5,4.5) -- (4.5,4.5) -- (4.5,0.5) -- (0.5,0.5) -- cycle;
\draw (10,-0.8) node[anchor=north west] {$C \frac{\log(1/q)}{q^2}$};
\draw [<-] (0,-3)--(9,-3); 
\draw [->] (20,-3)--(30,-3);
\draw [<->] (-1,0)--(-1,5);
\draw (-11,5) node[anchor=north west] {$C \frac{\log(1/q)}{q}$};
\draw (15,2.5) node {$B$};
\draw (15,11.5) node {$T_0$};
\draw (-8,11.5) node {$T_1^-$};
\draw (38,11.5) node {$T_1^+$};
\draw (2.4,2.5) node {$D$};
\end{tikzpicture}
\caption{The geometric setting  for the toy model of Figure \ref{fig:example}.}
\label{fig:mechanism:3}
\end{figure}
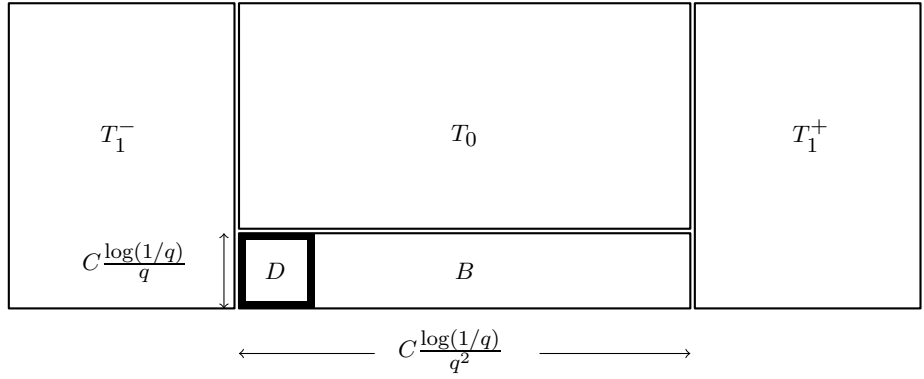

Let
$\O_0$ consist of all
configurations of $\{0,1\}^V$ such that: 
\begin{itemize}
\item each column of $B$ contains an infected site;
\item each row of $T_0$ contains a pair of adjacent infected sites;
\item each column of $T_1^\pm$ contains a pair of adjacent infected sites.
\end{itemize}
Notice that by choosing $C$ large enough $\mu(\O_0)=1-o(1)$ as $q\to 0.$ Then, in the key Theorem~\ref{thm:key step}, we prove that for any $f:\{0,1\}^V\to \bbR$
\[
\1_{\{D\text{ is infected}\}}\var_V(f\tc \O_0)\le
e^{O(\log(1/q)^3)/q}\times \cD(f),
\]
where $\cD(f)$ is the Dirichlet form of $f$ (see \eqref{eq:Dirichlet}). One can interpret the above inequality as saying that the KCM in $V$ restricted to the good set $\O_0$ has a relaxation time at most $e^{O(\log(1/q)^3)/q}$. We prove this by an inductive procedure over $T_0,T_1^\pm$ which, in some sense, makes rigorous the dynamic heuristics described above.

\section{Notation and preliminaries}
\label{sec:prelim}
In this section we gather the relevant notation and basic inputs from bootstrap percolation and KCM theories. We shall always denote spatial regions (either in $\bbZ^2$ or in $\bbR^2$) with capital letters and events in the various probability spaces with calligraphic capital letters.

\subsection{Bootstrap percolation}
\subsubsection{Stable and quasi-stable directions}
\label{sec:stable-dir}For every integer $n$, we write $[n]:=\{0,1,\dots, n-1\}$. We fix a critical update family $\cU$ with difficulty $\alpha=\alpha(\cU)$ and with a \emph{finite} set $\cS$ of stable directions.\footnote{By Lemmas~2.6 and 2.8 of~\cite{Bollobas14} this is equivalent to the fact that all (stable) directions have finite difficulty.} Using the definition of $\a(\cU)$ (see \eqref{eq:def:alpha}) one can fix an open semicircle $C$ with midpoint $u_0$, one of whose endpoints is in $\cS$ and such that $\max_{u\in C}\alpha(u)=\alpha$. Using~\cite{Bollobas15}*{Lemma~5.3} (see also~\cite{Bollobas14}*{Lemma~3.5} and~\cite{Martinelli19a}*{Lemma~4.6}) one can choose a set of rational directions\footnote{A direction $u\in S^1$ is rational if $\tan(u)\in\bbQ$ or, equivalently, if $su\in\bbZ^2\}$ for some $s>0$.} $\cS'\supset\cS$, so that for every two consecutive elements $u$ and $v$ of $\cS'$ there exists an update rule $X\in\cU$ such that $X\subset(\bbH_u\cup \ell_u)\cap(\bbH_v\cup \ell_v)$, where $\ell_{u'}=\{x\in \bbR^2:\< x,u'\>=0\}$ is the boundary of $\bbH_{u'}=\{x\in\bbR^2:\<x,u'\><0$ for any $u'\in S^1$. The elements of $\cS'$ are usually referred to as \emph{quasi-stable directions}. Then our fundamental set of directions will be
\begin{equation}
\label{eq:1}
\hS=\bigcup_{u\in\cS'}(\{u,u_0-(u-u_0)\}+\{0,\pi/2,\pi,3\pi/2\}).
\end{equation}
In other words, we start with the stable directions, add to them the quasi-stable ones, reflect them at $u_0$ and finally make the set obtained invariant by rotation by $\pi/2$. By construction the
cardinality of $\hS$ is a multiple of $4$. 
\begin{rem}
Let us note that invariance by rotation and reflection is cosmetic and one could in fact deal directly with the set of quasi-stable directions from~\cite{Bollobas15}, though notation would be more laborious and drawings less aesthetic.
\end{rem}

We write $u_0,u_1,\ldots u_{4k-1}$ for the elements of $\hS$ ordered
clockwise starting with $u_0$ and $\widehat\cS_0$ for those elements of $\hS$ belonging to the semicircle $C$. For all figures we shall take
$\hS=\{i\pi/4,i\in\{0,1,\ldots,7\}\}$ and $u_0= \pi$. When referring
to $u_i$, the index $i$ will be considered modulo $4k$. With this convention $\widehat\cS_0=\{u_{-k+1},\dots, u_{k-1}\}$.

\subsubsection{$\cU$-bootstrap percolation restricted to $\L\subset \bbZ^2$}
In the sequel, we will sometimes need the following slight variation of the $\cU$-bootstrap percolation. Given $\L\subset \bbZ^2$ and a set $A\subset \L$ of initial infection, we will write $[A]_\cU^\L$ for the closure $\bigcup_{t\ge 0} A_t^\L$ of the $\cU$-bootstrap percolation restricted to $\L,$ $(A_t^\L)_{t\ge 0},$ defined by
\[
A_{t+1}^\L= A_t^\L\cup\{x\in \L,\exists U\in \cU,  x+U\subset A_t^\L\}.
\]

\subsubsection{Geometric setup}
\label{sec:geosetup}We next turn to defining the various geometric domains we will need to
consider. As the notation is a bit cumbersome, the reader is invited
to systematically consult the relevant figures. We fix a large integer $w$ and a small positive number
$\d$ depending on $\cU$ (e.g.\ $w$
much larger than the diameter of $\cU$ and of the largest difficulty
of stable directions), but not depending on $q$. When using asymptotic
notation (as $q\to 0$) we will assume that the implicit constants do
not depend on $w,\d$ and $q$.
Throughout the entire
paper we shall consider that $q$ is small, as we are interested in the
$q\to 0$ limit. In particular, we shall assume that $ q$ is so small
that any length scale diverging to $+\infty$ as $q\to 0$ will be
(much) larger than
the constant $w$.

\begin{defn}
\label{def:polygon}
Consider a closed convex polygon $P$ in
$\bbR^2$. Assume that the outward normal vectors to
the sides of $P$ belong to $\hS$ and that $u$ is one of them. Then we write
$\partial_{u}P$ for the side whose outward normal is $u_i$.
\end{defn}

We can now define the notion of droplet that will be relevant for our
setting (see Figure~\ref{fig:ring}). In the sequel for $u\in S^1$ we set $\overline \bbH_{u}=\bbH_u\cup\ell_u$ for the closure of $\bbH_u$. Moreover, given $x\in\bbR^2$ and $s\in \bbR$, we set
$\bbH_{u}(x)=\bbH_{u}+x$, $\bbH_u(s)=\bbH_u(su)$ and
similarly for $\overline \bbH_u$ and $\ell_u$. Finally, for any $u_i\in \hat\cS$ we set $\r_i=\inf\{\r>0,\exists x\in\bbZ^2, \<x,u_i\>=\r\}$ for
the smallest positive $s$ such that $\ell_{u_i}(s)\neq \ell_{u_i}$ and
$\ell_{u_i}(s)\cap \bbZ^2\neq \varnothing$.
\begin{figure}
\begin{center}
\begin{tikzpicture}[line cap=round,line join=round,x=1.5cm,y=1.5cm]
\fill[fill=black,fill opacity=0.4] (-0.41,1) -- (-1,0.41) -- (-1,-0.41) -- (-0.41,-1) -- (0.41,-1) -- (1,-0.41) -- (1,0.41) -- (0.41,1) -- cycle;
\fill[line width=0pt,color=ffffff,fill=ffffff,fill opacity=1.0] (-0.35,0.85) -- (-0.85,0.35) -- (-0.85,-0.35) -- (-0.35,-0.85) -- (0.35,-0.85) -- (0.85,-0.35) -- (0.85,0.35) -- (0.35,0.85) -- cycle;
\fill[fill=black,pattern=horizontal lines] (-0.2,1) -- (-0.41,1) -- (-1,0.41) -- (-1,-0.41) -- (-0.41,-1) -- (-0.2,-1) -- (-0.85,-0.35) -- (-0.85,0.35) -- cycle;
\draw (-0.41,1)-- (-1,0.41);
\draw (-1,0.41)-- (-1,-0.41);
\draw (-1,-0.41)-- (-0.41,-1);
\draw (-0.41,-1)-- (0.41,-1);
\draw (0.41,-1)-- (1,-0.41);
\draw (1,-0.41)-- (1,0.41);
\draw (1,0.41)-- (0.41,1);
\draw (0.41,1)-- (-0.41,1);
\draw (-0.2,1)-- (-0.41,1);
\draw (-0.41,1)-- (-1,0.41);
\draw (-1,0.41)-- (-1,-0.41);
\draw (-1,-0.41)-- (-0.41,-1);
\draw (-0.41,-1)-- (-0.2,-1);
\draw (-0.2,-1)-- (-0.85,-0.35);
\draw (-0.85,-0.35)-- (-0.85,0.35);
\draw (-0.85,0.35)-- (-0.2,1);
\draw [->] (-1,0) -- (-1.2,0);
\draw [<->] (0.71,-0.71) -- (0.6,-0.6);
\draw [<->] (0,0) -- (0.6,0.6);
\begin{scriptsize}
\draw[color=black] (0.55,0.15) node {$R_3-w$};
\draw[color=black] (-1.1,0.1) node {$u_0$};
\draw[color=black] (0.8,-0.7) node {$w$};
\end{scriptsize}
\end{tikzpicture}
\end{center}
\caption{The shaded region is the \emph{quasi-stable annulus} $A$, while the
  hatched one is the \emph{quasi-stable half-annulus} $HA$. As anticipated all the
  radii $R_i$ are much larger than the width $w$.}
\label{fig:ring}
\end{figure}
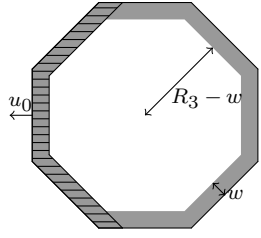
\begin{defn}[Quasi-stable annulus and half-annulus] 
\label{def:ring}

Fix a radius $R=R(q)$ such that $\lim_{q\to 0}R(q)=+\infty$ and let $R_i=\r_i\left\lfloor\frac{R}{\r_i}\right\rfloor$ for $i\in[4k]$. We call the subset of $\bbR^2$
\[
A=\bigcap_{i\in [4k]}\overline\bbH_{u_i}(R_i)\setminus \bigcap_{i\in [4k]}\bbH_{u_i}(R_i-w) 
\]
the \emph{quasi-stable annulus} (or simply the \emph{annulus}) with radius
$R$ and width $w$ centered at the origin. We write $\Aint$ for the
region $\bigcap_{i\in [4k]}\bbH_{u_i}(R_i-w)$ enclosed by
$A$. Clearly, the outer boundary of $A$ is a closed convex polygon $P$
satisfying the assumption of Definition \ref{def:polygon} and we 
write $\partial_{u_i} A$ for $\partial_{u_i} P$. We also let 
\begin{equation}
    \label{eq:boundaryC}
\partial_{\widehat \cS_0}A=\bigcup_{u\in\widehat \cS_0}\partial_{u} A,
\end{equation}
and we call 
\[HA=\bigcap_{i=-k}^{k}\overline\bbH_{u_i}(R_i)\setminus 
  \bigcap_{i=-k+1}^{k-1}\bbH_{u_i}(R_i-w)
\] 
the
\emph{quasi-stable half-annulus} of radius $R$ and width $w$.
\end{defn}

Our approach will consist in building progressively larger domains for which we can bound the Poincar\'e constant of the finite volume KCM process conditionally on the simultaneous occurrence of a certain likely event \emph{and} the presence of an infected annulus. We next define these domains (see Figure~\ref{fig:snails}).
Recall that $\d$ is a small constant depending on the update family $\cU$.
\begin{defn}[Snails]
\label{def:snail}
Recall $R$ and $R_i$ from Definition \ref{def:ring}. Let $L=L(q)>0$ be such that $\lim_{q\to 0}L(q)=+\infty$ and assume that 
$\frac{L\<u_0,u_{k-1}\>}{\r_{k-1}}\in\bbN$ (i.e.\ $\ell_{u_{k-1}}(Lu_0)$ contains lattice sites). We call a sequence of
non-negative numbers
$\ur=(r_0,r_1,\dots, r_{2k})$ \emph{admissible} if 
\begin{align*}
  0\le r_0&{}\le\d L,&r_i&{}\le \d r_{i-1},& r_{2k}&{}=0.
\end{align*}
Given an admissible
$\ur$ we call the set
\[
V_L^{R,+}(\ur) =
  \bigcap_{i=-k+1}^{k-1}\overline \bbH_{u_i}(R_i+L\<u_0,u_i\>) \cap
\bigcap_{i=k}^{3k}\overline \bbH_{u_i}(R_i+r_{i-k}) 
\] the \emph{right-snail}
with parameters $(R,L,\ur)$. Using the symmetric construction
of $\hS,$ the left-snail $V_L^{R,-}(\ur)$ with parameters
$(R,L,\ur)$ is simply defined as the reflection of the
right-snail w.r.t.\ the line orthogonal to $u_0$ and passing through the point $\frac 12 L u_0$. Finally the \emph{snail} with parameters $(R,L,\ur)$ is the set
\[
  V^R_L(\ur)=V_L^{R,+}(\ur)\cup V_L^{R,-}(\ur).
\]
We systematically
drop the parameters $R$, $L$, and $\ur$ from our notation when no ambiguity arises.
\end{defn}
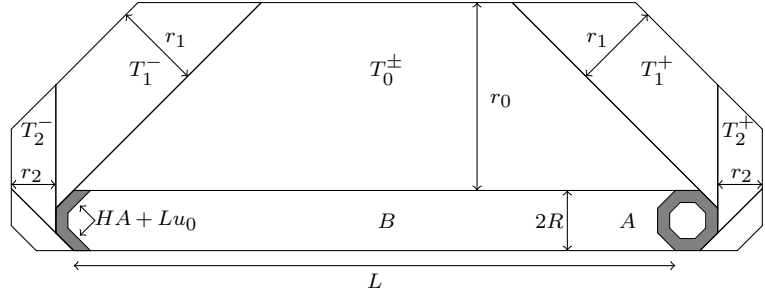
\begin{figure}
\begin{center}
\begin{tikzpicture}[line cap=round,line join=round,x=0.4cm,y=0.4cm]
\clip(-25,-3) rectangle (5,7.3);
\fill[fill=black,fill opacity=0.5] (-0.41,1) -- (0.41,1) -- (1,0.41) -- (1,-0.41) -- (0.41,-1) -- (-0.41,-1) -- (-1,-0.41) -- (-1,0.41) -- cycle;
\fill[color=ffffff,fill=ffffff,fill opacity=1.0] (-0.25,0.6) -- (0.25,0.6) -- (0.6,0.25) -- (0.6,-0.25) -- (0.25,-0.6) -- (-0.25,-0.6) -- (-0.6,-0.25) -- (-0.6,0.25) -- cycle;
\fill[color=ffffff,fill=ffffff,fill opacity=1.0] (-20.25,0.6) -- (-19.75,0.6) -- (-19.4,0.25) -- (-19.4,-0.25) -- (-19.75,-0.6) -- (-20.25,-0.6) -- (-20.6,-0.25) -- (-20.6,0.25) -- cycle;
\fill[fill=black,fill opacity=0.5] (-19.85,1) -- (-20.41,1) -- (-21,0.41) -- (-21,-0.41) -- (-20.41,-1) -- (-19.85,-1) -- (-20.6,-0.25) -- (-20.6,0.25) -- cycle;
\draw (-0.41,1)-- (0.41,1);
\draw (0.41,1)-- (1,0.41);
\draw (1,0.41)-- (1,-0.41);
\draw (1,-0.41)-- (0.41,-1);
\draw (0.41,-1)-- (-0.41,-1);
\draw (-0.41,-1)-- (-1,-0.41);
\draw (-1,-0.41)-- (-1,0.41);
\draw (-1,0.41)-- (-0.41,1);
\draw [color=ffffff] (-0.25,0.6)-- (0.25,0.6);
\draw [color=ffffff] (0.25,0.6)-- (0.6,0.25);
\draw [color=ffffff] (0.6,0.25)-- (0.6,-0.25);
\draw [color=ffffff] (0.6,-0.25)-- (0.25,-0.6);
\draw [color=ffffff] (0.25,-0.6)-- (-0.25,-0.6);
\draw [color=ffffff] (-0.25,-0.6)-- (-0.6,-0.25);
\draw [color=ffffff] (-0.6,-0.25)-- (-0.6,0.25);
\draw [color=ffffff] (-0.6,0.25)-- (-0.25,0.6);
\draw (-0.6,0.25)-- (-0.6,-0.25);
\draw (-0.6,-0.25)-- (-0.25,-0.6);
\draw (-0.25,-0.6)-- (0.25,-0.6);
\draw (0.25,-0.6)-- (0.6,-0.25);
\draw (0.6,-0.25)-- (0.6,0.25);
\draw (0.6,0.25)-- (0.25,0.6);
\draw (0.25,0.6)-- (-0.25,0.6);
\draw (-0.25,0.6)-- (-0.6,0.25);
\draw (-20.41,1)-- (-21,0.41);
\draw (-21,0.41)-- (-21,-0.41);
\draw (-20.41,-1)-- (-21,-0.41);
\draw (-20.41,1)-- (-14.17,7.25);
\draw (-14.17,7.25)-- (-5.83,7.25);
\draw (-5.83,7.25)-- (0.41,1);
\draw (0.41,1)-- (-20.41,1);
\draw (-5.83,7.25)-- (-1.73,7.25);
\draw (-1.73,7.25)-- (1,4.51);
\draw (1,4.51)-- (1,0.41);
\draw (1,0.41)-- (-5.83,7.25);
\draw (1,4.51)-- (2.48,3.04);
\draw (2.48,3.04)-- (2.48,1.06);
\draw (2.48,1.06)-- (1,-0.41);
\draw (1,-0.41)-- (1,4.51);
\draw (2.48,1.06)-- (2.48,-0.17);
\draw (2.48,-0.17)-- (1.65,-1);
\draw (1.65,-1)-- (0.41,-1);
\draw (0.41,-1)-- (2.48,1.06);
\draw (-20.41,-1)-- (-0.41,-1);
\draw (-14.17,7.25)-- (-18.27,7.25);
\draw (-18.27,7.25)-- (-21,4.51);
\draw (-21,4.51)-- (-21,0.41);
\draw (-21,0.41)-- (-14.17,7.25);
\draw (-21,4.51)-- (-22.48,3.04);
\draw (-22.48,3.04)-- (-22.48,1.06);
\draw (-22.48,1.06)-- (-21,-0.41);
\draw (-21,-0.41)-- (-21,4.51);
\draw (-22.48,1.06)-- (-22.48,-0.17);
\draw (-22.48,-0.17)-- (-21.65,-1);
\draw (-21.65,-1)-- (-20.41,-1);
\draw (-20.41,-1)-- (-22.48,1.06);
\draw (-19.85,1)-- (-20.41,1);
\draw (-20.41,1)-- (-21,0.41);
\draw (-21,0.41)-- (-21,-0.41);
\draw (-21,-0.41)-- (-20.41,-1);
\draw (-20.41,-1)-- (-19.85,-1);
\draw (-19.85,-1)-- (-20.6,-0.25);
\draw (-20.6,-0.25)-- (-20.6,0.25);
\draw (-20.6,0.25)-- (-19.85,1);
\draw [->] (-19.7,0)--(-20.2,0.5);
\draw [->] (-19.7,0)--(-20.2,-0.5);
\draw [<->] (-7,1)--(-7,7.25);
\draw [<->] (-3.38,4.8)--(-1.33,6.85);
\draw [<->] (-16.62,4.8)--(-18.67,6.85);
\draw [<->] (1,1.2)--(2.48,1.2);
\draw [<->] (-21,1.2)--(-22.48,1.2);
\draw [<->] (-20.41,-1.5)--(-0.41,-1.5);
\draw [<->] (-4,-1)--(-4,1);
\begin{scriptsize}
   \draw (-18,0) node {$HA+Lu_0$};
\draw (-2,0) node {$A$};
   \draw (-10,5) node {$T_0^\pm$};
   \draw (-10,0) node {$B$};
   \draw (-1,5) node {$T_1^+$};
   \draw (-18,5) node {$T_1^-$};
   \draw (-21.6,3) node {$T_2^-$};
   \draw (1.7,3) node {$T_2^+$};
   \draw (-6.2,4) node {$r_0$};
   \draw (-3,6) node {$r_1$};
   \draw (-17,6) node {$r_1$};
   \draw (1.8,1.5) node {$r_2$};
   \draw (-21.8,1.5) node {$r_2$};
   \draw (-10.4,-2) node {$L$};
   \draw (-4.6,0) node {$2R$};
\end{scriptsize}
\end{tikzpicture}
\end{center}
\caption{A snail $V=V^+\cup V^-$ with its base $B$ and its trapezoids $T_0^\pm, T_1^\pm,\dots$. The right-snail of $V$ is $V^+= B\cup \bigcup_i T_i^+$ while the left-snail is $V^-= B\cup \bigcup_i T_i^-$. In Section \ref{sec:core} the shaded quasi-stable annulus $A$ and the half-annulus $HA+Lu_0$ will act as an infected boundary condition.}
\label{fig:snails}
\end{figure}
\begin{defn}
\label{def:ti}
We observe that any right-snail $V^{R,+}_L(\ur)$ can be thought
of as the set obtained by 
stacking together as in Figure~\ref{fig:snails} its \emph{base} defined as
\[B=V^{R}_L((0,\dots,0))\]
and its trapezoids defined as 
\begin{align*}
T^+_i={}&V^{R,+}_L(r_0,\dots,r_i,0,\dots,0)\setminus V^{R,+}_L(r_0,\dots,r_{i-1},0,\dots,0)\\
={}&
(\overline\bbH_{u_{k+i}}(R_{k+i}+r_i)\setminus\overline\bbH_{u_{k+i}}(R_{k+i}))
\cap\overline\bbH_{u_{k+i-1}}(R_{k+i-1}+r_{i-1})\cap\overline\bbH_{u_{k+i+1}}(R_{k+i+1})
\end{align*}
with the convention $r_{-1}=L\<u_0,u_{k-1}\>$. Notice that the base $B$ is characterized by two parameters $R,L$ called \emph{radius} and \emph{length} respectively.\end{defn}
With this picture in mind the positive values of $\ur$
coincide with the heights of the corresponding non-empty trapezoids. A similar decomposition holds for the left-snail. In the sequel, it will be convenient to partition the lattice sites in each trapezoid $T_i^+$ into disjoint \emph{slices} $ST^+_{i,j}, j=1,2,\dots,$ with each slice consisting of all the lattice sites of the trapezoid lying on a \emph{common line} of $\bbR^2$ orthogonal to the direction $u_{k+i}$. Similarly for the lattice sites contained in the \emph{truncated base} $B^\circ := B\setminus (A\cup \Aint\cup (HA+Lu_0))$. In this case each slice, denoted $SB_j,j=1,2,\dots,$ will consist of all the sites belonging to a common suitable translate in the $u_0$-direction of $\partial_{\widehat \cS_0}A$ defined in \eqref{eq:boundaryC}. Recall from Definition \ref{def:ring} $\r_i$ and $R_i=\rho_i\lfloor R/\rho_i\rfloor$, where $R$ is the radius of the annulus $A$. 
\begin{defn}
\label{def:snail2} Fix a snail and suppose that its trapezoid $T_i^+$, $i\in [2k]$ is non-empty. The $j$\textsuperscript{th} slice of $T_i^+\cap \bbZ^2$, in the sequel $ST^+_{i,j},$ is the set $ST^+_{i,j}=T_i^+\cap\ell_{u_{k+i}}(R_{k+i}+j\r_{k+i})\cap\bbZ^2$ in such a way that
$T_i^+\cap\bbZ^2=\bigcup_{j>0}ST^+_{i,j}$. Similarly for the left trapezoid $T_i^-$ if non-empty.

Turning to the truncated base $\oB$, we first set $\l_{j+1}=\inf\{\l>\l_j,(\l u_0+\partial_{\hS_0}A)\cap\bbZ^2\neq\varnothing\}$ with $\l_0=0$. Then we define the $j$\textsuperscript{th} slice of the truncated base $\oB$ of the snail, $SB_j,$ and its $i$\textsuperscript{th}-side, $SB_{i,j},$ as \begin{align*}SB_j&{}=(\l_j+\partial_{\hS_0}A)\cap\oB\cap\bbZ^2,\\
SB_{i,j}&{}=(\l_j+\partial_{u_i}A)\cap\oB\cap\bbZ^2.
\end{align*}
\end{defn}
Note that for any admissible sequence
$\ur$
a non-empty slice of the trapezoid $T^+_i$ consists of all lattice points of a segment $I\subset \bbR^2$ orthogonal to $u_{k+i}$ with length $\O(r_{i-1})$ and such that $I\cap \bbZ^2\neq \emptyset$. Similarly, the number of lattice sites in each slice of $\oB$ is $\Theta(R)$. In the sequel we will only consider non-empty slices without explicitly specifying the range of the index $j>0$.

\subsubsection{Helping sets}
Recall Definition \ref{def:stable:alpha} and Remark \ref{rem:finite difficulty}. If $u$ is a stable direction, then the infected half-plane $\bbH_u$ needs \emph{finitely} many (exactly $\a(u)$) extra infected sites in $\bbR^2\setminus \bbH_u$ in order to infect \emph{infinitely} many sites on the line $\ell_u$. If only a finite portion of $\bbH_u$ is infected, e.g.\ the dashed region in Figure~\ref{fig:lem:extension}, then the propagation of infection to some portion of the line $\ell_u$ is a delicate problem. A special case which suffices for our purposes is covered in the next lemma (see \cite{Bollobas14}*{Lemma~3.4} and \cite{Bollobas15}*{Lemma~5.2}).

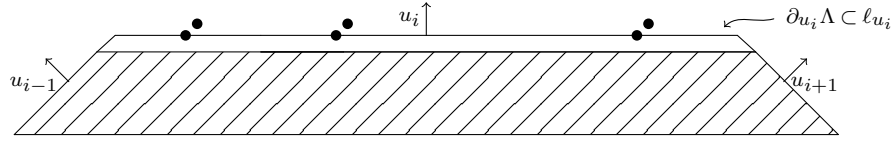
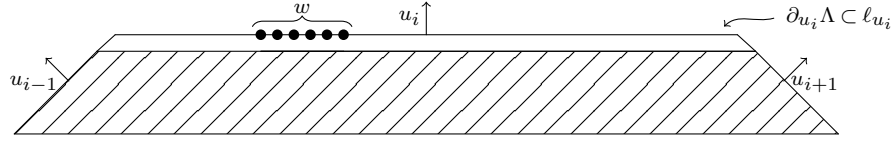
\begin{figure}
\begin{subfigure}{\textwidth}
\centering\begin{tikzpicture}[line cap=round,line join=round,x=2.2cm,y=2.2cm]
\fill[pattern=my north east lines] (1.99,6.5) -- (2.49,6) -- (-2.49,6) -- (-1.99,6.5) -- cycle;
\fill[fill=white] (-0.5,6.6) -- (-0.5,6.5) -- (1.99,6.5) -- (1.88,6.6) -- cycle;
\draw (-1,6.6)-- (-0.5,6.6);
\draw (-0.5,6.5)-- (-1,6.5);
\draw (1.99,6.5)-- (2.49,6);
\draw (2.49,6)-- (-2.49,6);
\draw (-2.49,6)-- (-1.99,6.5);
\draw (-1.99,6.5)-- (1.99,6.5);
\draw (-0.5,6.5)-- (1.99,6.5);
\draw (1.99,6.5)-- (1.88,6.6);
\draw (1.88,6.6)-- (-0.5,6.6);
\draw [->] (-2.16,6.32) -- (-2.3,6.46);
\draw [->] (0,6.6) -- (0,6.8);
\draw [->] (2.16,6.32) -- (2.3,6.46);
\draw (-0.5,6.5)-- (-1,6.5);
\draw (-1.88,6.6)-- (-1.99,6.5);
\draw (-1.88,6.6)-- (-1,6.6);
\foreach \x [count=\n] in {-1,1,5}{
  \begin{scope}[xshift = \x cm]
    \fill (-1,6.6) circle (2pt);
    \fill (-1+0.07,6.6+0.07) circle (2pt);
    \end{scope}
}
\begin{scriptsize}
\draw (-2.37,6.3) node {$u_{i-1}$};
\draw (-0.1,6.7) node {$u_i$};
\draw (2.5,6.7) node {$\partial_{u_i} \L\subset \ell_{u_i}$};
\draw (2.35,6.3) node {$u_{i+1}$};
\end{scriptsize}
\draw[<-] (1.8,6.65) to [out=45,in=195] (2.1,6.7) ;
\end{tikzpicture}
\caption{\label{subfig:w}An example helping set (the black dots) consisting of three disjoint copies of $Z$ shifted along $\ell_{u_i}$. In the figure $\a(u_i)=2$ with $Z=\{(0,0),(1,1)\}$ and $m=3$.}
\end{subfigure}

\begin{subfigure}{\textwidth}
\centering
\begin{tikzpicture}[line cap=round,line join=round,x=2.2cm,y=2.2cm]
\fill[pattern=my north east lines] (1.99,6.5) -- (2.49,6) -- (-2.49,6) -- (-1.99,6.5) -- cycle;
\fill[fill=white] (-0.5,6.6) -- (-0.5,6.5) -- (1.99,6.5) -- (1.88,6.6) -- cycle;
\draw (-1,6.6)-- (-0.5,6.6);
\draw (-0.5,6.5)-- (-1,6.5);
\draw (1.99,6.5)-- (2.49,6);
\draw (2.49,6)-- (-2.49,6);
\draw (-2.49,6)-- (-1.99,6.5);
\draw (-1.99,6.5)-- (1.99,6.5);
\draw (-0.5,6.5)-- (1.99,6.5);
\draw (1.99,6.5)-- (1.88,6.6);
\draw (1.88,6.6)-- (-0.5,6.6);
\draw [->] (-2.16,6.32) -- (-2.3,6.46);
\draw [->] (0,6.6) -- (0,6.8);
\draw [->] (2.16,6.32) -- (2.3,6.46);
\draw (-0.5,6.5)-- (-1,6.5);
\draw (-1.88,6.6)-- (-1.99,6.5);
\draw (-1.88,6.6)-- (-1,6.6);
\fill (-1,6.6) circle (2pt);
\fill (-0.5,6.6) circle (2pt);
\fill (-0.9,6.6) circle (2pt);
\fill (-0.8,6.6) circle (2pt);
\fill (-0.7,6.6) circle (2pt);
\fill (-0.6,6.6) circle (2pt);
\begin{scriptsize}
\draw[decorate,decoration={brace,amplitude=4pt}] (-1.05,6.65)--(-0.45,6.65) node [black,midway,yshift=7] {$w$};
\draw (-2.37,6.3) node {$u_{i-1}$};
\draw (-0.1,6.7) node {$u_i$};
\draw (2.5,6.7) node {$\partial_{u_i} \L\subset \ell_{u_i}$};
\draw (2.35,6.3) node {$u_{i+1}$};
 \end{scriptsize}
\draw[<-] (1.8,6.65) to [out=45,in=195] (2.1,6.7) ;
\end{tikzpicture}
\caption{Illustration of $w$ consecutive sites of $\partial_{u_i}\L$.}
\end{subfigure}
\caption{
The setting of Lemma~\ref{lem:strip}. In the figure $u_i$ is the upwards direction and the hatched trapezoid represents the lattice sites in $\L\setminus \partial_{u_i}\L$. The lemma states that if the hatched region and the black sites are infected, $\partial_{u_i}\L$ also becomes infected (in the $\cU$ bootstrap percolation process restricted to a suitable region).}
\label{fig:lem:extension}
\end{figure}
\begin{lem}\label{lem:strip}
Fix $u=u_i$, $i\in[4k]$ and recall that $w$ is a large enough integer (depending on $\cU$) and let $r\ge w^2$. Let $\L:=\L(u,w,r)=\overline \bbH_{u_{i-1}}(r)\cap\overline\bbH_{u_{i}}\cap\overline \bbH_{u_{i+1}}(r)\cap\overline\bbH_{u_{i+2k}}(w)$
be the (closed) trapezoid in Figure~\ref{fig:lem:extension} of height $w$ and bases orthogonal to $u$. Note that $\partial_u \L\subset \ell_u$.
\begin{enumerate}[label=(\alph*),ref=(\alph*)]
\item\label{case:a} Let $Z \subset \bbZ^2\setminus \bbH_u$ be a set of $\a(u)$ sites at distance at most $\sqrt{w}$ from the origin such that $[\bbH_u \cup Z]_\cU \cap \ell_u$ is infinite. Then there exist finitely many lattice points $a_1,\dots,a_m, b,$ on the line $\ell_u$ such that the following holds. If $\L\setminus \partial_u \L$ \emph{and} $\bigcup_{j=1}^m (Z+a_j+k_jb)$ are infected, where $k_1,\dots, k_m \in \bbZ$ are such that $\{a_1 + k_1 b, a_2+k_2 b,\dots, a_m+k_m b\}$ form $m$ distinct lattice sites of $\partial_u \L$ at distance at least $w$ from the endpoints of $\partial_u \L,$   then the $\cU$-bootstrap percolation restricted to the larger trapezoid $\overline \L= \overline \bbH_{u_{i-1}}(r)\cap\overline\bbH_{u_i}(w/2)\cap\overline\bbH_{u_{i+1}}(r)\cap\overline\bbH_{u_{i+2k}}(w)
$ is able to infect $\partial_u\L$.
\item\label{item:w:consecutive} If $\L\setminus \partial_u \L$ and $w$ consecutive lattice sites in $\partial_u \L$ are infected, then the $\cU$-bootstrap percolation restricted to $\L$ is able to infect $\partial_u\L$.
\end{enumerate}
\end{lem}
\begin{defn}[$u$-helping sets]
\label{defn:helping set}Let $i\in[-k+1,k-1]$. Any collection of lattice sites of the form $\{a_1 + k_1 b, a_2+k_2 b,\dots, a_m+k_m b\}$ satisfying the assumption in \ref{case:a} above will be referred to as \emph{$u_i$-helping set} for $\partial_{u_i}\L$ or simply $u_i$-helping set.
\end{defn}

\subsection{Some KCM tools}
\label{sec:tools}
For reader's convenience we next collect some general tools from KCM theory that will be applied several times throughout the proof of the main result.
\subsubsection{Notation}
For every statement $\cP$ define $\1_{\{\cP\}}=1$ if $\cP$ holds and $\1_{\{\cP\}}=0$ otherwise. For any subset $\L$ of $\bbR^2$ we write $(\O_\L,\mu_\L)$ for the product probability space $(\{0,1\}^{\L\cap\bbZ^2},\bigotimes_{x\in {\L\cap\bbZ^2}} \mu_q)$. If $\L=\bbR^2$, we simply write $(\O,\mu)$. Given $f:\O_\L\to \bbR$ we shall write $\mu_\L(f)$ and $\var_\L(f)$ for the mean and variance of $f$ w.r.t.\ $\mu_\L$ respectively whenever they exist. For any $\o\in \O$ and $\L\subset \bbR^2$ we write $\o_\L$ for the collection $(\o_x)_{x\in \L\cap\bbZ^2}$. Given a function $f:\O\to \bbR$ depending on finitely many variables we write 
\begin{equation}
  \label{eq:Dirichlet}
  \cD(f)=\sum_{x\in \bbZ^2}\mu(c_x \var_x(f)),
\end{equation}
for the KCM \emph{Dirichlet form} of $f$, where $c_x(\o)$ is the indicator of the event $\{\exists X\in \cU: \forall y\in X,\o_{x+y}=0\}$ and $\var_x(f):=\var_{\{x\}}(f)$ denotes the conditional variance $\var(f\tc (\o_z)_{z\neq x})$. Finally, we shall write $\bbP_\mu(\cdot)$ for the law of the $\cU$-KCM process on $\bbZ^2$ with initial law $\mu$ and $\bbE_\mu(\cdot)$ for the expectation w.r.t.\ $\bbP_\mu(\cdot)$.

\subsubsection{Poincar\'e inequalities}

We begin with a well-known general fact on product measures which we state here in ready-to-use form.
\begin{lem}
\label{lem:var:conv}
Let $\L_i$, $i\in\{1,2,3\}$ be three disjoint finite subsets of $\bbZ^2$ and $\nu_i$ be a probability measures on $\O_{\L_i}$. Let $\nu$ be the product measure $\bigotimes_{i=1}^3\nu_i$ on $\bigotimes_{i=1}^3\O_{\L_i}$. Then for any function $f$ we have
\[\nu_1(\var_{\nu_2\otimes\nu_3}(f))\le \var_\nu(f)\le \nu_1(\var_{\nu_2\otimes \nu_3}(f))+\nu_2(\var_{\nu_1\otimes\nu_3}(f)).\]
\end{lem}
\begin{proof}[Proof of Lemma \ref{lem:var:conv}]
The first inequality follows from the total variance formula
\[\var_\nu(f)= \nu_{1}(\var_{\nu_2\otimes\nu_3}(f))+\var_{\nu_1}(\nu_2\otimes\nu_3(f)).\]
For the second inequality we observe that 
\begin{align*}
\var_{\nu_1}(\nu_2\otimes\nu_3(f))&=\nu_{1}((\nu_2\otimes\nu_3(f-\nu(f))^2)\\
&\le \nu((f-\nu_1\otimes\nu_3(f))^2)=\nu_2(\var_{\nu_1\otimes\nu_3}(f))
\end{align*}
by Jensen's inequality.
\end{proof}
In order to understand the general framework for the last two results, we begin by  recalling a standard Poincar\'e inequality for $n$ independent random variables $X_1,\dots,X_n$ (for simplicity each one taking finitely many values). For any $f=f(X_1, \dots, X_n)$ 
\[
\var(f)\le \sum_i \bbE(\var_i(f)),
\]
where $\var_{i}(f)$ is the conditional variance computed w.r.t.\ the variable $X_i$ given all the other variables. 
The sum in the r.h.s.\ above can be interpreted as the Dirichlet form of the continuous time \emph{Gibbs sampler}, reversible w.r.t.\ the product law of $(X_i)_i$, which with rate $n$ chooses a random index $i\in [n]$ and resamples $X_i$ w.r.t.\ its marginal. From this perspective, the above inequality tells us that the relaxation time (see e.g.\ \cite{Levin09}) of the Gibbs sampler is bounded from above by 1.

Now consider $n$ events $(\cH_i)_{i=1}^n$, in the sequel \emph{facilitating events}, and suppose that each $\cH_i$ depends only on the variables $(X_j)_{j\neq i}$. An example of a \emph{constrained} Poincar\'e inequality with facilitating events $(\cH_i)_{i=1}^n$ is the inequality \begin{equation}
\label{eq:framework1}
\var(f\tc \O_\cH)\le C \sum_{i}\frac{\bbE(\1_{\cH_i}\var_{i}(f))}{\bbP(\O_\cH)},    
\end{equation}
where $\O_\cH=\bigcup_{i=1}^n \cH_i$ and $C\in [1,+\infty].$ Notice that the sum in the r.h.s.\ above can be interpreted as the Dirichlet form of the continuous time \emph{constrained} Gibbs sampler on $\O_\cH$, which with rate $n$ chooses a random index $i\in [n]$ and resamples $X_i$ w.r.t.\ its marginal iff $\cH_i$ holds. If the facilitating events are such that the constrained Gibbs sampler on $\O_\cH$ is ergodic then $C<+\infty$. 

Each one of the two results we are about to discuss next is just a special instance of the above general problem.

\begin{lem}
\label{lem:2block}
Let $X_1,X_2$ be two independent random variable taking values in two finite sets $\bbX_1,\bbX_2$.  Let also $\cH\subset \bbX_{1}$ with $\bbP(X_1\in \cH)>0$. Then for any function $f(X_1,X_2)$ 
\[
\var(f)\le  2\bbP(X_1\in \cH)^{-1}\bbE\left(\var_{1}(f)
+\1_{\{X_1\in \cH \}}
\var_{2}(f)\right). 
\]
\end{lem}
\begin{rem}
The above inequality coincides with \eqref{eq:framework1} with $\cH_1=\bbX_2$, $\cH_2= \{X_1\in\cH\}$ and $C= 2/\bbP(X_1\in \cH)$. Clearly the constrained Gibbs sampler is irreducible because $\bbP_1(X_1\in \cH)>0$.
\end{rem}
\begin{proof}[Proof of Lemma \ref{lem:2block}]
  It follows from \cite{Cancrini08}*{Proof of Proposition 4.4} that
  \begin{align}
\var(f)&\le  \frac{1}{1-\sqrt{1-\bbP(X_1\in\cH)}}\bbE\left(\var_{1}(f)
+\1_{\{X_1\in \cH \}}
             \var_{2}(f)\right)\nonumber\\
    &\le
2\bbP(X_1\in \cH)^{-1}\bbE\left(\var_{1}(f)
+\1_{\{X_1\in \cH \}}
\var_{2}(f)\right).\tag*{\qedhere}
   \end{align}
\end{proof}
The second result concerns a generalisation of the standard (finite volume) constrained Poincar\'e inequality for the $1$-neighbour KCM process, or FA1f KCM, \cite{Cancrini08}.

Let $(\widehat S,\widehat \nu)$ be a finite probability space with $\widehat \nu$ a positive probability measure, let $\O_n=\widehat S^{[n]}$ and $\nu=\bigotimes_{i\in[n]}\nu_i$, where $\nu_i=\widehat\nu$ for
all $i\in [n]$. Elements of $\O_n$ are denoted $\o=(\o_{0},\dots,\o_{n-1})$
with $\o_i\in \widehat S$. Fix a single site event
$\cH\subset \widehat S$ and a positive integer $\k<n$. Then, according
to whether we view the set $[n]$ as the $n$-cycle or not, we define
the \emph{facilitating event} $\cH_i$ as
follows. If $[n]$ is the $n$-cycle
\[
\cH_i=
\bigcap_{j=i+1}^{i+\k}\{\o_j\in \cH\}\cup\bigcap_{j=i-1}^{i-\k}\{\o_j\in \cH\}.
\]
If instead $[n]$ is linear
\[
\cH_i=\begin{cases}
\bigcap_{j=i-1}^{i-\k}\{\o_j\in \cH\}  & \text{if $i+\k\ge n$}\\
\bigcap_{j=i+1}^{i+\k}\{\o_j\in \cH\} & \text{if $i-\k< 0$}\\
\bigcap_{j=i+1}^{i+\k}\{\o_j\in \cH\}\cup \bigcap_{j=i-1}^{i-\k}\{\o_j\in \cH\} & \text{otherwise.}
\end{cases}
\]
In words, in the periodic case $\cH_i$ requires the $\k$ variables immediately after or before $i$ (in the clockwise order) to be in a state belonging to $\cH$, while in the linear case the same requirement holds when $i$ is farther than $\k$ from the boundary points of $[n]$. When $i$ is closer than $\k$ to e.g.\ the the right boundary of $[n]$, then $\cH_i$ requires the $\k$ variables immediately before $i$ to be in states belonging to $\cH$. The case when $\widehat S=\{0,1\}$, $\widehat \nu$ is the Bernoulli$(1-q)$-measure,  $\cH=\{0\}$ and $\k=1$ is the usual $1$-neighbour KCM setting.

\begin{lem}
  \label{lem:FAperiodic}
Assume that $(1-\widehat \nu(\cH)^k)^{n/(3\k)}<\frac{1}{16}$ and set $\O_\cH=\bigcup_{i=0}^{n-1} \cH_i$. Then, for all
  $f:\O_n\to \bbR$
\begin{equation}\label{eq:FAeasy}
\var_\nu(f\tc \O_\cH) \le \left(\frac{2}{\widehat\nu(\cH)}\right)^{O(\k)}\sum_{i = 1}^n \nu\left(\1_{\cH_i} \var_{\nu_i}(f) \right).
\end{equation}
\end{lem}
The proof is left to the \hyperref[app]{Appendix}.
\begin{rem}We will apply the lemma with $(\widehat S,\widehat \nu)$ equal to the probability space given by $\{0,1\}^m$ equipped with the $\text{Bernoulli}(1-q)$ product measure conditioned on some event whose probability tends to one as $q\to 0$. The integers $1\ll m\ll n$ may diverge to infinity as $q\to 0$ while the integer $\k$ will be large but independent of $q$.
\end{rem}

\section{The core of the proof}
\label{sec:core}
In this section we prove a Poincar\'e inequality which will
represent the key step in the proof of Theorem
\ref{th:main}.

\subsection{Roadmap}
\label{subsec:roadmap}
Before we dive into the technical details, let us give a hands-on roadmap of the argument. Although it is underlied by the dynamical intuition explained in Section \ref{sec:sketch}, the latter is not very transparent in the Poincar\'e language of the formal proof. 

The goal of this section is to prove Theorem \ref{thm:key step}. It says that the $\cU$-KCM ($\cU$ being a fixed critical update family with a finite number of stable directions) on a snail $V=V_L^R(\ur)$ (recall Definition \ref{def:snail} and Figure \ref{fig:snails}), conditioned on a well-chosen \emph{super good} event $\SG(V)$ is able to relax in a time $\exp(\log^3(1/q)/q^\a)$, which is the dominating contribution leading to \eqref{eq:1.4}. For the purposes of the roadmap the reader should think of the snail as having dimensions $R=w^2\log(1/q)/q^\a$, $L=q^{-3w}$ and $r_{i}=\d^{i-(2k-1)}q^{-2w}, i\in[2k]$ for some small positive $\d$. Let us explain the Definition \ref{def:good events} of $\SG(V)$ before outlining the proof of Theorem \ref{thm:key step}. 

\subsubsection{Good and super good events}
\label{sec:good-supergood}
The super good event $\SG(V)$ will decompose as a product w.r.t.\ the partition of $V$ into its annulus $A$, half-annulus $HA+Lu_0$, annulus interior $\Aint$, truncated base $\oB$ and trapezoids $T^\pm_i$ from Section \ref{sec:geosetup}. On $A$ ($HA$) we require the event $\cA$ ($\HA$) that $A$ ($HA$) is fully infected. These are the only \emph{unlikely} events involved in $\SG(V)$ and we will denote by $\SG$ only events requiring the occurrence of (spatial translates of) $\cA$ and $\HA$. Events of type $\SG$ will all have very small probability $\m(\SG)$ of the order of $\exp(-\log^2(1/q)/q^\a)$. 

We will use instead write $\cG$ to denote good events, which are \emph{likely} and only involve the presence of appropriate helping sets as in Definition \ref{defn:helping set} or sets of $w$ consecutive infections as in Lemma \ref{lem:strip}\ref{item:w:consecutive}. Recall the decomposition of $\oB$ into slices $SB_j$ from Definition \ref{def:snail2}. We say that the event $\SB_j$ occurs if each side of $SB_j$ (which consists of at most one segment in each direction) has a helping set for the corresponding direction. We then define $\cG(\oB)=\bigcap_j\SB_j$ and it is not hard to see that this way the occurrence of $\SG(B)=\cA\cap\HA\cap\cG(\oB)$ implies that the infections in $B$ are sufficient to fully infect $B$.

Notice that in general the event $\SB_j$ depends on the values of $\o$ in the set $\bigcup_{i=0}^k SB_{j+i}$ for some $k\ge0$ depending only on $\cU$. In order to avoid this (annoying) technical detail we will use the following simplifying Assumption \ref{ass:1} implying that $k=0.$ 
\begin{ass}
\label{ass:1}
For every stable direction $u\in \widehat\cS_0$ there exists a subset $Z_u$ of the line $\ell_u$ of cardinality $\a$ such that $[\bbH_u\cup Z_u]_\cU\cap\ell_u$ has infinite cardinality.   
\end{ass}
This is by no means restrictive, as the proof applies directly without this assumption up to changing $\SB_{i,j}$ in Definition~\ref{def:good events}, following \cite{Martinelli19a}*{Sec. 7}. We will spare the reader the tedious details, as they already appeared previously in the above-mentioned paper. This assumption is only relevant for treating the base $B$, for which we will import the result from \cite{Martinelli19a}, where the assumption was introduced.

Having defined the good event for the base $B$, we now define the good event for the trapezoids of the snail $V$. Let $\ST^\pm_{i,j}$ be the event that the slice $ST^\pm_{i,j}$ in the decomposition of $T^\pm_i$ from Definition \ref{def:snail2} contains a set of $w$ consecutive infected sites. We then define $\cG(T^\pm_i)=\bigcap_j\ST^\pm_{i,j}$. Again, by Lemma \ref{lem:strip} it is not hard to see that if $B$ and $T^+_{i'}$ for $i'<i$ are fully infected and $\cG(T^+_i)$ occurs, then the $\cU$-bootstrap percolation can also infect $T^+_i$ (and similarly for $T_i^-$). 

Finally, the super good event $\SG(V)$ is defined as $\SG(B)\cap\bigcap_{i}(\cG(T^+_i)\cap\cG(T^-_i))$ and it clearly implies that the entire snail $V$ can be infected from within.

\subsubsection{Structure of the proof}
The fact that $V$ can be fully infected on $\SG(V)$ is reassuring and implies that the relaxation time we are after in Theorem \ref{thm:key step} is finite, but we need an efficient relaxation mechanism to prove the  theorem. It is not hard to see that it suffices to treat the right-snail $V^+$, so we concentrate on it and drop all $+$ superscripts. 
In the sequel, whenever we refer to the \emph{relaxation in a given region $\L$} mathematically this will translate into proving a Poincar\'e inequality like the one in \eqref{eq:3} with $V$ replaced by $\L$. 

The proof proceeds by proving an efficient relaxation in progressively larger and larger volumes always conditioned on a corresponding $\SG$ event. In the process we will often rely on auxiliary constrained block dynamics of several types like those in Section \ref{sec:tools}. These auxiliary dynamics allow us to relate the relaxation in a given region to the relaxation in smaller sub-regions, each subregion having an additional convenient constraint on the configuration outside it. The auxiliary dynamics we will use are of FA1f type (like the one in Lemma \ref{lem:FAperiodic}) or two-blocks type (like the one in Lemma \ref{lem:2block}). By performing such reductions, we reduce the problem of proving an efficient relaxation on a large region to a similar problem on suitable smaller regions. The base case of the above inductive procedure is then treated directly. We now describe the various steps of the above iterative reduction.

\paragraph{The base case: the annulus interior $\Aint$}
First, in Lemma \ref{lem:Aint} we treat $\Aint$ on the event $\cA$ that the annulus is fully infected, which serves as a boundary condition. This is fairly easy and can be done in various ways. To give a formal argument, we split $\Aint$ into strips of bounded width (see Figure \ref{fig:Aint}). Fully infected strips perform an FA1f auxiliary dynamics. The boundary condition provides all the sets of $w$ consecutive infections needed for an infected strip to infect its neighbour using Lemma \ref{lem:strip} (see Figure \ref{subfig:w}).

\paragraph{From $\Aint$ to the base $B$}
Up to now we have a Poincar\'e inequality on the annulus and its interior. In Proposition \ref{prop:base case} we extend that to a base $B$. We will not insist on this step, as it was essentially done already in \cite{Martinelli19a}. 
Indeed, using an East-like dynamics in direction $u_0$ the relaxation time of $\oB$ (on $\cG(\oB)$) with infected boundary condition in $A$ was shown to be roughly $\exp(\log^3(1/q)/q^\a)$. Combining this with the result for $\Aint$, we obtain a Poincar\'e inequality for $B$.

\paragraph{Adding the first trapezoid to $B$}
Our next goal is to consider the relaxation in $B\cup T_0$. In turn, this step is split into two distinct parts.

\subparagraph{Adding the first slice $ST_{0,1}$ to $B$} 
This is achieved in Lemma \ref{prop:slice} (see Figure \ref{fig:slice}). Relaxation in $B$ has already been established in the previous step, so we focus on the relaxation in $ST_{0,1}$. In doing this we are allowed our knowledge of the relaxation in $B$. We use the FA1f-like dynamics of Lemma \ref{lem:FAperiodic}, asking for $w$ consecutive infections in $ST_{0,1}$ next to the site to be updated. In other words we have to understand how to efficiently resample a site $x\in ST_{0,1}$ using the $\cU$-KCM dynamics when its neighbouring $w$ sites are infected. Using a two-block dynamics (Lemma \ref{lem:2block}), resampling $B$ roughly $q^{-O(1)}$ times, we may further impose the condition that the site we wish to resample has a fully infected neighbourhood in $B$ in addition to the next $w$ sites in $ST_{0,1}$, which are already infected. This is exactly the situation in Figure \ref{fig:slice} and makes the flip of the site we want to update legal for the original $\cU$-KCM. Thus, this step produces terms of the Dirichlet form of the $\cU$-KCM in \eqref{eq:Dirichlet}, as well as a term $\var_{B}(f\tc\SG(B))$, which we already know how to control.

\subparagraph{Adding more slices to $B$}
In a sense this part embodies the East-like motion of droplets in direction $u_1$ hinted in Section \ref{sec:sketch}. This connection is rather indirect in the sense that the \emph{bisection method} used to analyse the relaxation in the union of $B$ with several slices of the first trapezoid coincides with the bisection method used to efficiently bound from above the relaxation time of the standard East model in \cite{Cancrini08}. 

Consider the problem of the relaxation in a snail consisting of $B$ and $2n$ slices of the first trapezoid. Our aim is to reduce it to the same problem on two similar snails, each one with essentially the same base $B$ but with only $n$ slices. This is achieved in Lemma \ref{prop:bisection}. We start by introducing an auxiliary constrained two-block dynamics in which $B$ and the first $n$ slices form the first block $\widetilde V,$ while the second group of $n$ slices form the second block $\overline T_0$ (see Figure \ref{fig:shift}). The constraint of the two-blocks dynamics is that a translated base $\overline B$ (corresponding to $\overline V_{i-1}$ in Figure \ref{fig:shift}) is super good. The base $\overline B$ is constructed so that together with $\overline T_0$ it forms a snail $\overline V$ with size similar to that of $\widetilde V$. The relaxation to equilibrium on $\widetilde V$ is dealt with by induction on the number of slices, so it remains to analyse the relaxation to equilibrium on $\overline T_0$ under the above constraint. The relaxation time of the auxiliary model is $1/\m(\SG)$ (the number of times one needs to update the first block until the constraint becomes satisfied). Then in order to relax on $\overline T_0$ it suffices to do so on the larger region $\overline V$. We are done since $\overline V$ and $\widetilde V$ are already treated by the induction.

In Corollary \ref{cor:trapezoid}, repeating the above bisection several times, we manage to reproduce the relaxation on a snail with base $B$ and arbitrary number $r_0$ of slices of the first trapezoid. Indeed, starting from the snail with a single slice in $T_0$ provided above, we double its height $\log (1/q)$ times to reach the desired $r_0\simeq \d^{-2k+1}q^{-2w}$. Thus, the Poincar\'e constant of $B$ is multiplied by $1/\m(\SG)^{\log (1/q)}$ in this process.

\paragraph{Adding all trapezoids of the original snail $V$}
Finally, repeating the above steps for each trapezoid, we obtain the desired Poincar\'e constant for the entire snail, concluding Theorem \ref{thm:key step}. 

\subsection{Setup}
Given a snail $V=V_L^R(\ur)$, we shall work in the associated probability space $\O_{V}=\{0,1\}^{V\cap \bbZ^2}$ endowed with the probability measure $\mu_V(\cdot\tc \SG(V))$ conditioned to the simultaneous occurrence of the following events on $\O_{V}$.

\begin{defn}[Good and super good events]
  \label{def:good events}
\leavevmode\begin{itemize}
\item Recalling Definition \ref{def:ring}, we define $\cA$ as the event that $A$ is infected and $\HA$ as the event that $HA+Lu_0$ is infected.
\item Recalling Definitions \ref{def:snail2} and \ref{defn:helping set}, for each $SB_j$ and $u_i\in\widehat\cS_0$, let $\SB_{i,j}$ denote the event that $SB_{i,j}=\varnothing$ or $SB_{i,j}$ contains an infected $u_i$-helping set. Then set $\SB_j=\bigcap_{u_i\in\widehat\cS_0}\SB_{i,j}$.
\item Recalling Definition \ref{def:ti}, for each non-empty $ST^\pm_{i,j}$ let $\ST^\pm_{i,j}$ be the event that there exist $w$ consecutive infected sites in $ST^\pm_{i,j}$.
\end{itemize}
Using the above events, we then define
\begin{align*}
\cG(\oB)&{}=\bigcap_{j>0} \SB_{j},& \SG(B)&{}=\cA\cap\HA\cap \cG(\oB),& \cG(T^\pm_i)&=\bigcap_{j>0} \ST^\pm_{i,j}.
\end{align*}
Finally, we set $\SG(V^+)=\SG(B)\cap \bigcap_{i\in[2k]}\cG(T_i^+)$ and  $\SG(V)=\SG(V^+)\cap \SG(V^-),$
with $\SG(V^-)$ the analog of $\SG(V^+)$ for the left-snail.
\end{defn}
We note that the event $\HA$ is there only to ensure the easy removal of the simplifying  Assumption~\ref{ass:1}.
\begin{rem}
\label{rem:prod}The events above are defined so as to preserve as much as possible the original product structure of $\mu$ in the conditional measure $\mu_{V}(\cdot\tc \SG(V))$. In fact,
\begin{align*}
\mu_{V^\pm}(\cdot\tc \SG(V^\pm))&= \mu_B(\cdot\tc\SG(B))\otimes\left(\bigotimes_{i\in[2k]} \mu_{T^\pm_i}(\cdot \tc \cG(T^\pm_i))\right),\\
\mu_{T^\pm_i}(\cdot \tc \cG(T^\pm_i))&= \bigotimes_{j>0} \mu_{ST^\pm_{i,j}}(\cdot\tc \ST^\pm_{i,j}),\\
\mu_{B}(\cdot\tc \SG(B))&=\mu_{\Aint}\otimes\d_{\o_{A\cup(HA+Lu_0)}=0}\otimes\mu_{\oB}(\cdot\tc \cG(\oB)),\\
\mu_{\oB}(\cdot\tc\cG(\oB))&=\bigotimes_{j>0}\mu_{SB_j}(\cdot\tc\SB_j),
\end{align*}
since trapezoids and the base are pairwise disjoint by construction and likewise for the slices of the trapezoids, the slices of the base, the annulus, its interior and the translated half-annulus.
\end{rem}
Taking into account this product structure, in the next observations we establish that, as claimed in Section \ref{subsec:roadmap}, all $\cG$ events we will use are likely and all $\SG$ events have roughly the same probability, $q^{\Theta(Rw)}$.
\begin{obs}
\label{obs:prob:G}
Let $R\ge w^2\log(1/q)/q^\a$ and $L\le q^{4w}$, let $\ur$ be admissible (see Definition \ref{def:snail}) and $r_{i-1}\ge q^{-2w}$ for some $i\in[2k]$. Then $\mu(\cG(T_i^\pm))\ge 1-o(1)$ and $\mu(\cG(\oB))\ge 1-o(1)$.
\end{obs}
\begin{proof}
For the first assertion notice that the condition implies that for all $j>0$, $ST^\pm_{i,j}$ is either empty or has cardinality at least $\O(q^{-2w})$. Then by Remark \ref{rem:prod}
\[\mu(\cG(T^\pm_i))=\prod_j\mu(\ST^\pm_{i,j})\ge\left(1-(1-q^w)^{\O(q^{-2w}/w)}\right)^{O(r_i)}\ge 1-o(1),\]
since $r_i\le L\le q^{4w}$ by Definition \ref{def:snail} and by assumption.

The second assertion is proved similarly (see e.g.\ \cite{Martinelli19a}*{Lemma 6.5}).
\end{proof}
\begin{obs}
\label{obs:prob:SG}
Let $R\ge w^2\log(1/q)/q^\a$, $L\le q^{-4w}$, let \ur be admissible such that for some $i\in[2k]$, $r_{i+1}=0$ and $r_{i-1}\ge q^{-2w}$ with the convention $r_{-1}=L$. Then
\begin{equation}
\label{eq:muG}
\mu(\SG(V_L^R(\ur)))=q^{\Theta(Rw)}.
\end{equation}
\end{obs}
\begin{proof}
Using Remark \ref{rem:prod} and Observation \ref{obs:prob:G}, it suffices to note that
\[\mu(\cA\cap\HA)\ge q^{|A|+|HA|}=q^{\Theta(Rw)}.\tag*{\qedhere}\]
\end{proof}

\subsection{Key step}
We are ready to state the main result of this
section. In the sequel, for any $\L\subset \bbZ^2,$ any $x\in \L$ and any $\o_\L\in \O_\L$ we shall write $c^\L_x(\o_\L)$ for the constraint $c_x(\o)$ computed for the configuration $\o$ equal to $\o_\L$ in $\L$ and equal to $1$ elsewhere.  By construction, $c_x^\L(\o_\L)\le c_x(\o')$ for any $\o'\in \O$ such that $\o'_\L=\o_\L$ and $c_x^\L\ge c_x^{\L'}$ for any $\L'\subset\L$. Then for any snail $V$ (or base) we write $\g_V$ for the smallest constant $\g\in[1,\infty]$ such that the Poincar\'e inequality
    \begin{equation}
      \label{eq:3}
     \var_{V}(f\tc \SG(V))\le \g
      \sum_{x\in V}\mu_{V}\left(c^V_x\var_x(f)\right)
    \end{equation}
holds for every function $f:\O\to\bbR$. 
\begin{thm}
    \label{thm:key step}
    There exist $w_0,\d_0>0$ not depending on $q$ such that for any $0<\d\le\d_0$ and $w\ge w_0$ the following holds for any $R=\Theta(w^2\log (1/q)/q^\a)$. Consider the snail $V=V_L^R(\ur)$ for admissible $L,\ur$ such that $r_{2k-1}\ge q^{-2w}$ and $L\le q^{-4w}$. Then
\begin{equation}
  \label{eq:4}
\g_V\le e^{-O(w^4\log^3(1/q)/q^\a)}.
\end{equation}
 \end{thm}

\begin{claim}
We have $\g_{V}\le 3\max(\g_{V^+},\g_{V^-})$.
\end{claim}
\begin{proof}
Set $\L_1=V^+\setminus(B\cup T_0^+)$, $\L_2=V^-\setminus(B\cup T_0^+)$, $\L_3=B\cup T_0^+=B\cup T_0^-$, $\nu_1=\mu_{\L_1}(\cdot\tc \bigcap_{i=1}^{2k-1}\cG(T^+_i))$, similarly for $\nu_2$ and $\nu_3=\mu_{\L_3}(\cdot\tc\SG(B)\cap \cG(T_0^+))$. By Remark \ref{rem:prod} we can apply Lemma~\ref{lem:var:conv} to obtain
\begin{align*}\var_V(f\tc\SG(V))&\le \begin{multlined}[t]
\g_{V^-}\sum_{x\in V^-}\nu_1\left(\mu_{V^-}\left(c^{V^-}_x\var_x(f)\right)\right)\\+\g_{V^+}\sum_{x\in V^+}\nu_2\left(\mu_{V^+}\left(c^{V^+}_x\var_x(f)\right)\right)\end{multlined}\\
&\le (1+o(1))(\g_{V^-}+\g_{V^+})\sum_{x\in V}\mu\left(c^V_x\var_x(f)\right),
\end{align*}
where in the last inequality we used Observation \ref{obs:prob:G} to remove the conditioning of $\nu_1$ and $\nu_2$.
\end{proof}

Therefore, in order to prove \eqref{eq:4} it suffices to prove the analogous statement with $V$ replaced by $V^\pm$. In the sequel we will concentrate on proving  \eqref{eq:4} for the best constant $\g_{V^+}$ in the Poincar\'e inequality \eqref{eq:3} with $V$ replaced by its right-snail $V^+$. The proof is based on comparison methods between Markov processes and induction over right-snails with different $L$ and $\ur$ as outlined in Section \ref{subsec:roadmap}. If we exchange right-snails with left-snails the same proof will then apply to the left-snail $V^-$ as well. Since our arguments no longer require a left-snail, for lightness of notation, we drop the superscript ``$+$'' from our notation whenever possible.

The proof of the theorem is decomposed into two quite different steps (see Propositions \ref{prop:base case} and \ref{prop:snail:reduction} below). In the first one, labeled the \emph{base case}, we consider a right-snail $V$ with no trapezoids (\ur=0). In the second step, labeled \emph{reduction step}, roughly speaking we compare the Poincar\'e constant $\g_V$ of a generic right-snail $V$ with the same constant computed for its base $B$.

The conclusion of Theorem \ref{thm:key step} follows at once from \eqref{eq:muG}, Proposition \ref{prop:base case} and Proposition \ref{prop:snail:reduction}. In the sequel fix $\d,w,R$ as in the statement of the theorem and recall that $B=V_L^R(0)$.

\begin{rem}
\label{rem:log1}
For future purposes (see the discussion in Section  \ref{sec:open pb}) it is very important to emphasise that it is only in the first step that we use directly the definition of the event $\SG(B)$ entering in the event $\SG(V)$ (cf.\ Definition \ref{def:good events}). In the second step  the only property of the event $\SG(B)$ that is needed is that it is a decreasing event in $\O_{B}$ w.r.t.\ the partial order $\o\prec \o'$ iff $\o_x\le \o'_x$ for all $x\in B$. 
\end{rem}
\subsection{Base case} 
\begin{prop}
\label{prop:base case}
For any $f:\O_{B}\to \bbR$
\[\var_{B}(f\tc \SG(B))\le q^{-O(Rw\log L)}\sum_{x\in B}\mu_{B}\left(c^B_x\var_x(f)\right).\]
\end{prop}
\begin{proof}[Proof of Proposition \ref{prop:base case}]
We first observe that, up to minor modifications, in \cite{Martinelli19a}*{Proposition 6.6}  it was proved that for all $f:\O_{\oB}\to \bbR$
\begin{equation}
  \label{eq:base1}
 \1_{\cA\cap\HA} \var_{\oB}(f\tc \cG(\oB)))\le q^{-O(Rw\log L)} \1_{\cA} \sum_{x\in\oB}\mu_{\oB}\left(c^B_x\var_x(f)\right).
\end{equation}
The next step in the proof is an analogous result for $\Aint$.
\begin{lem}
\label{lem:Aint}
For any $f:\O_{\Aint}\to \bbR$
\[
\1_{\cA}\var_{\Aint}(f)\le
  q^{-O(Rw)}\1_{\cA}\sum_{x\in\Aint}\mu_{\Aint}\left(c^{A\cup\Aint}_x\var_x(f)\right).\]
\end{lem}
\begin{figure}
    \centering
    \begin{tikzpicture}[line cap=round,line join=round,>=triangle 45,x=2.5cm,y=2.5cm]
\clip(-1,-1) rectangle (1,1);
\fill[fill=black,fill opacity=0.5] (-0.41,1) -- (-1,0.41) -- (-1,-0.41) -- (-0.41,-1) -- (0.41,-1) -- (1,-0.41) -- (1,0.41) -- (0.41,1) -- cycle;
\fill[line width=0pt,color=white,fill=white,fill opacity=1.0] (-0.35,0.85) -- (-0.85,0.35) -- (-0.85,-0.35) -- (-0.35,-0.85) -- (0.35,-0.85) -- (0.85,-0.35) -- (0.85,0.35) -- (0.35,0.85) -- cycle;
\fill[pattern=my north east lines] (-0.85,0.35) -- (-0.8,0.4) -- (-0.8,-0.4) -- (-0.85,-0.35) -- cycle;
\fill[pattern=my north east lines] (-0.65,0.55) -- (-0.65,-0.55) -- (-0.5,-0.7) -- (-0.5,0.7) -- cycle;
\fill[pattern=my north east lines] (-0.35,0.85) -- (-0.2,0.85) -- (-0.2,-0.85) -- (-0.35,-0.85) -- cycle;
\fill[pattern=my north east lines] (-0.05,0.85) -- (0.1,0.85) -- (0.1,-0.85) -- (-0.05,-0.85) -- cycle;
\fill[pattern=my north east lines] (0.4,-0.8) -- (0.35,-0.85) -- (0.25,-0.85) -- (0.25,0.85) -- (0.35,0.85) -- (0.4,0.8) -- cycle;
\fill[pattern=my north east lines] (0.55,-0.65) -- (0.7,-0.5) -- (0.7,0.5) -- (0.55,0.65) -- cycle;
\draw (-0.41,1)-- (-1,0.41);
\draw (-1,0.41)-- (-1,-0.41);
\draw (-1,-0.41)-- (-0.41,-1);
\draw (-0.41,-1)-- (0.41,-1);
\draw (0.41,-1)-- (1,-0.41);
\draw (1,-0.41)-- (1,0.41);
\draw (1,0.41)-- (0.41,1);
\draw (0.41,1)-- (-0.41,1);
\draw (-0.85,0.35)-- (-0.8,0.4);
\draw (-0.8,0.4)-- (-0.8,-0.4);
\draw (-0.8,-0.4)-- (-0.85,-0.35);
\draw (-0.85,-0.35)-- (-0.85,0.35);
\draw (-0.65,0.55)-- (-0.65,-0.55);
\draw (-0.65,-0.55)-- (-0.5,-0.7);
\draw (-0.5,-0.7)-- (-0.5,0.7);
\draw (-0.5,0.7)-- (-0.65,0.55);
\draw (-0.35,0.85)-- (-0.2,0.85);
\draw (-0.2,0.85)-- (-0.2,-0.85);
\draw (-0.2,-0.85)-- (-0.35,-0.85);
\draw (-0.35,-0.85)-- (-0.35,0.85);
\draw (-0.05,0.85)-- (0.1,0.85);
\draw (0.1,0.85)-- (0.1,-0.85);
\draw (0.1,-0.85)-- (-0.05,-0.85);
\draw (-0.05,-0.85)-- (-0.05,0.85);
\draw (0.4,-0.8)-- (0.35,-0.85);
\draw (0.35,-0.85)-- (0.25,-0.85);
\draw (0.25,-0.85)-- (0.25,0.85);
\draw (0.25,0.85)-- (0.35,0.85);
\draw (0.35,0.85)-- (0.4,0.8);
\draw (0.4,0.8)-- (0.4,-0.8);
\draw (0.55,-0.65)-- (0.7,-0.5);
\draw (0.7,-0.5)-- (0.7,0.5);
\draw (0.7,0.5)-- (0.55,0.65);
\draw (0.55,0.65)-- (0.55,-0.65);
\begin{scriptsize}
\draw (-0.73,0) node {$K_{2}$};
\draw (-0.43,0) node {$K_{4}$};
\end{scriptsize}
\end{tikzpicture}
    \caption{Setting of the proof of Lemma \ref{lem:Aint}. Every second strip $K_i$ of $\Aint$ is hatched. The annulus $A$ is shaded.}
    \label{fig:Aint}
\end{figure}
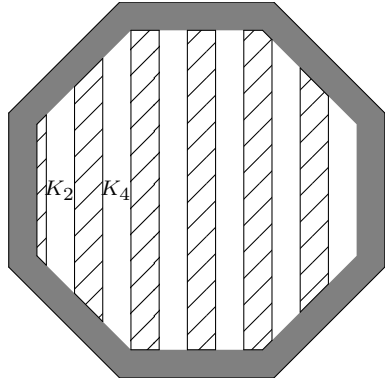
\begin{proof}[Proof of Lemma \ref{lem:Aint}]
Let us partition $\Aint$ into disjoint strips $K_i$ of width $w$ perpendicular to $u_{0}$ and number them from left to right (see Figure \ref{fig:Aint}).

We can then apply \cite{Martinelli19a}*{Proposition 3.4} on the generalised FA1f KCM to obtain
\[
\var_{\Aint}(f)\le
q^{-O(Rw)}\sum_{i}\mu((\1_{\cH_i^+}+\1_{\cH_i^-})\var_{K_i}(f))
,\]
where $\cH^\pm_i$ are the events that $K_{i\pm1}$ is fully infected and we use the convention that $\cH^+_i$ occurs for the last strip and $\cH^-_i$ does for the first one, which corresponds to the boundary condition provided by $A$. W.l.o.g.\ it then suffices to bound the generic term $\mu(\1_{\cH^+_i}\var_{K_i}(f))$. But this can be done using Lemma 5.2 of \cite{Martinelli19a} and Lemma \ref{lem:strip}\ref{item:w:consecutive}, which guarantees that if $A$ and $K_{i+1}$ are infected, then $K_i$ can also be infected by the $\cU$-bootstrap percolation restricted to $K_i\cup K_{i+1}\cup A$.
\end{proof}
Using Lemma~\ref{lem:var:conv} with $\L_1=\oB$, $\L_2=\Aint$, $\L_3=\emptyset$, $\nu_1=\mu_{\oB}(\cdot\tc \cG(\oB))$, and $ \nu_2=\mu_\Aint$, we obtain
\begin{multline*}
\1_{\cA\cap\HA}\var_{\oB\cup\Aint}(f\tc \cG(\oB))\\
\le\1_{\cA\cap\HA}\left(\mu_\Aint\left(\var_{\oB}(f\tc \cG(\oB))\right)+\mu_{\oB}\left(\var_{\Aint}(f)\tc \cG(\oB)\right)\right).
\end{multline*}
The first term in the r.h.s.\ above is bounded by 
\[q^{-O(Rw\log L)} \1_{\cA\cap\HA}\sum_{x\in\oB}\mu_{\oB\cup \Aint}(c^B_x\var_x(f)),\]
using  \eqref{eq:base1}, while the second one is bounded by 
\[q^{-O(Rw\log L)}\1_{\cA\cap\HA}\sum_{x\in\Aint}\mu_{\oB\cup\Aint}\left(c^{A\cup\Aint}_x\var_x(f)\tc \cG(\oB)\right).
\]
by Lemma \ref{lem:Aint}. By Remark \ref{rem:prod} we immediately get that
\begin{align*}
  \var_{B}(f\tc \SG(B))= {}&\mu_B(\var_{\oB\cup\Aint}(f\tc \cG(\oB))\tc \cA\cap\HA)\\
  \le{}& \mu(\SG(B))^{-1}q^{-O(Rw\log L)}\sum_{x\in B}\mu_{B}(c^B_x\var_x(f))
\end{align*}
and the proposition follows from Observation \ref{obs:prob:SG}.
\end{proof}
\subsection{Reduction step} 
Before we can state a relationship between $\g_V$ and $\g_B$, we need the following notion, which will cover all snail shapes that may arise during the reduction. Recall that $\d$, $w$, $V=V^{R,+}_L(\ur)$ are fixed as in the statement of Theorem \ref{thm:key step} and that we do not write the $+$ index, though all snails we refer to are right-snails. Also recall that all snails are defined by admissible sequences (see Definition \ref{def:snail}).
\begin{defn}
\label{def:relevant}
Let $C$ be a constant chosen sufficiently large depending on $\widehat \cS$, but much smaller than $1/\d_0$ in Theorem \ref{thm:key step}. We say that a snail $\widehat V=V^R_{\widehat L}(\widehat\ur)$ is \emph{of type $i\in[2k]$} if 
\begin{enumerate}[label=(\alph*),ref=(\alph*)]
    \item\label{item:type:i} $\widehat r_{i+1}=0$,
    \item\label{item:ri:smaller} $\widehat r_i\le r_i$,
    \item\label{item:r:difference} for all $j< i$ it holds that $0\le r_j-\widehat r_j\le C\left(r_i-\widehat r_i+\sum_{l=i+1}^{2k-1}r_l\right),$
    \item\label{item:L:difference} $0\le L-\widehat L\le C\left(r_i-\widehat r_i+\sum_{l=i+1}^{2k-1}r_l\right)$.
\end{enumerate}
We say that $\widehat V$ is \emph{relevant} if there exists $i\in[2k]$ such that $\widehat V$ is of type $i$. In particular, a base $\widehat B=V^R_{\widehat L}(0)$ is relevant iff $0\le L-\widehat L=O(r_0)$.
\end{defn}
In words, $\widehat V$ is relevant if all trapezoids except the last one are only slightly shorter than the corresponding ones for $V$ and similarly for the base, while the last trapezoid may be as much shorter as needed. Indeed, observe that by admissibility $\sum_{l=i+1}^{2k-1}r_l< 2 r_{i+1}$ for any $i\in[2k]$. 

Let us mention that the technical second inequalities in conditions \ref{item:r:difference} and \ref{item:L:difference} in the definition above are only needed for the inductive procedure below to always yield relevant snails. We invite the reader to ignore those conditions and admit that all smaller snails arising in our argument have sizes which can be treated by induction.

\begin{prop}
\label{prop:snail:reduction}
Let $\s=1/\min_{\widehat V}\mu_{\widehat V}(\SG(\widehat V))$ and $\G=\max_{\widehat B}\g_{\widehat B}$, where the $min$ and $\max$ run over relevant snails and relevant bases respectively. Then
\[\g_V\le\left(q^{-w^4}\s\right)^{O(\log L)}\G.
\]
\end{prop}

\begin{figure}
\centering
\begin{tikzpicture}[line cap=round,line join=round,>=triangle 45,x=0.17cm,y=0.17cm]
\clip(-71,-1) rectangle (1,26);
\fill[fill=black,fill opacity=0.5] (-0.41,1) -- (0.41,1) -- (1,0.41) -- (1,-0.41) -- (0.41,-1) -- (-0.41,-1) -- (-1,-0.41) -- (-1,0.41) -- cycle;
\fill[fill=white,fill opacity=1.0] (-0.25,0.6) -- (0.25,0.6) -- (0.6,0.25) -- (0.6,-0.25) -- (0.25,-0.6) -- (-0.25,-0.6) -- (-0.6,-0.25) -- (-0.6,0.25) -- cycle;
\fill[fill=black,fill opacity=0.5] (-69.85,1) -- (-70.41,1) -- (-71,0.41) -- (-71,-0.41) -- (-70.41,-1) -- (-69.85,-1) -- (-70.6,-0.25) -- (-70.6,0.25) -- cycle;
\fill[fill=black,fill opacity=0.5] (-17.5,18.91) -- (-22.58,18.91) -- (-8.61,4.94) -- (-8.61,10.03) -- cycle;
\draw (-0.41,1)-- (0.41,1);
\draw (0.41,1)-- (1,0.41);
\draw (1,0.41)-- (1,-0.41);
\draw (1,-0.41)-- (0.41,-1);
\draw (0.41,-1)-- (-0.41,-1);
\draw (-0.41,-1)-- (-1,-0.41);
\draw (-1,-0.41)-- (-1,0.41);
\draw (-1,0.41)-- (-0.41,1);
\draw (-0.25,0.6)-- (0.25,0.6);
\draw (0.25,0.6)-- (0.6,0.25);
\draw (0.6,0.25)-- (0.6,-0.25);
\draw (0.6,-0.25)-- (0.25,-0.6);
\draw (0.25,-0.6)-- (-0.25,-0.6);
\draw (-0.25,-0.6)-- (-0.6,-0.25);
\draw (-0.6,-0.25)-- (-0.6,0.25);
\draw (-0.6,0.25)-- (-0.25,0.6);
\draw (-0.6,0.25)-- (-0.6,-0.25);
\draw (-0.6,-0.25)-- (-0.25,-0.6);
\draw (-0.25,-0.6)-- (0.25,-0.6);
\draw (0.25,-0.6)-- (0.6,-0.25);
\draw (0.6,-0.25)-- (0.6,0.25);
\draw (0.6,0.25)-- (0.25,0.6);
\draw (0.25,0.6)-- (-0.25,0.6);
\draw (-0.25,0.6)-- (-0.6,0.25);
\draw (-70.41,1)-- (-71,0.41);
\draw (-71,0.41)-- (-71,-0.41);
\draw (-70.41,-1)-- (-71,-0.41);
\draw (-70.41,1)-- (-45.62,25.79);
\draw (-45.62,25.79)-- (-24.38,25.79);
\draw (-24.38,25.79)-- (0.41,1);
\draw (0.41,1)-- (-70.41,1);
\draw (-70.41,-1)-- (-0.41,-1);
\draw (-69.85,1)-- (-70.41,1);
\draw (-70.41,1)-- (-71,0.41);
\draw (-71,0.41)-- (-71,-0.41);
\draw (-71,-0.41)-- (-70.41,-1);
\draw (-70.41,-1)-- (-69.85,-1);
\draw (-69.85,-1)-- (-70.6,-0.25);
\draw (-70.6,-0.25)-- (-70.6,0.25);
\draw (-70.6,0.25)-- (-69.85,1);
\draw (-71,0.41)-- (-71,-0.41);
\draw (-71,-0.41)-- (-70.41,-1);
\draw (-70.41,-1)-- (0.41,-1);
\draw (0.41,-1)-- (1,-0.41);
\draw (1,-0.41)-- (1,0.41);
\draw (1,0.41)-- (-24.38,25.79);
\draw (-24.38,25.79)-- (-45.62,25.79);
\draw (-45.62,25.79)-- (-71,0.41);
\draw (-22.68,25.79)-- (1,2.11);
\draw (-17.5,18.91)-- (-22.58,18.91);
\draw (-22.58,18.91)-- (-8.61,4.94);
\draw (-8.61,4.94)-- (-8.61,10.03);
\draw (-8.61,10.03)-- (-17.5,18.91);
\begin{scriptsize}
\draw (-11,16.5) node [scale=1.5] {$x$};
\draw (-60,21) node [scale=1.5] {$\widetilde V$};
\draw (-1,8.32) node [scale=1.5] {$\L_x$};
\draw (-10,22.32) node [scale=1.5] {$\widehat{ST}_{i,1}$};
\draw[decorate,decoration={brace,amplitude=3pt}] (-10.7,15.2)--(-8.3,12.8) node [black,midway,xshift=7,yshift=7, scale=1.5] {$w$};
\end{scriptsize}
\draw (-12.21,15.32) node[cross=2pt,rotate=45] {};
\fill (-11.21,14.32) circle (2pt);
\fill (-10.21,13.32) circle (2pt);
\fill (-9.21,12.32) circle (2pt);
\draw [->] (-57,21)--(-51,21);
\draw [->] (-15,22.32)--(-19.21,22.32);
\draw [->] (-3,8.32)--(-8.61,8.32);
\end{tikzpicture}
    \caption{The geometric setting of Lemma \ref{prop:slice}. The snail is $\widetilde V$, while $\widehat{ST}_{i,1}=\widehat V\setminus\widetilde V$ is the remaining slice on the top-right. The site $x$ to be updated in \eqref{eq:generic:term} is marked by a cross. The event $\widetilde\cH_x$ corresponds to the shaded trapezoid $\L_x$ being infected and the event $\cH_x$ corresponds to the $w$ consecutive sites next to $x$ on one of its sides being infected.}
    \label{fig:slice}
\end{figure}
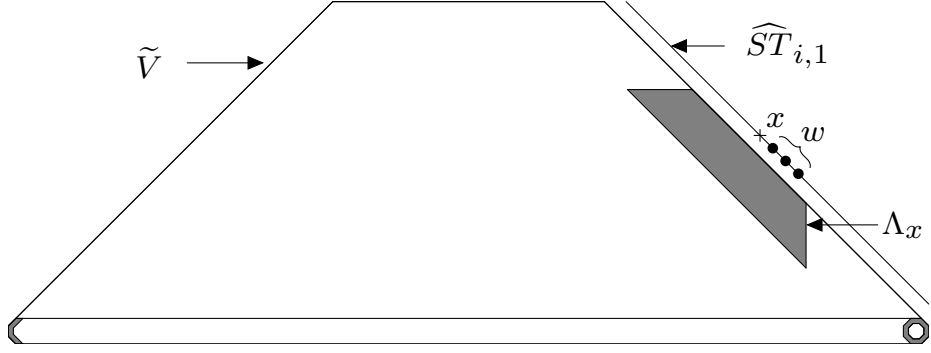

In the rest of the section we slowly build the proof of this proposition. The first step of reduction consists in removing a trapezoid consisting of a single slice. This is done using Lemma \ref{lem:FAperiodic} and may be intuitively understood as an FA1f dynamics of $w$ consecutive infected sites in the slice. Recall the definition of $\rho_i$ from Section \ref{sec:stable-dir}.
\begin{lem}[Removing a single slice]
\label{prop:slice}
Let $\widehat V=V^R_{\widehat L}(\widehat\ur)$ be a snail of type $i$ such that $\widehat r_i=\l\r_{k+i}$ for $\l\in\bbN$. In other words, the last non-empty trapezoid, $\widehat T_i$ of $\widehat V$ consists of $\l$ segments orthogonal to $u_{i+k}$. Then, setting $\widetilde \ur=(\widehat r_0,\dots,\widehat r_{i-1},0,\dots,0)$, $\widetilde V=V^R_{\widehat L}(\widetilde\ur)$, we have

\[\g_{\widehat V}\le \left(q^{-w^4}\right)^{O(\l)}\max\left(\g_{\widetilde V},1/\mu_{\widehat V}(\SG(\widehat V))\right).\]
\end{lem}
\begin{proof}[Proof of Lemma \ref{prop:slice}]
By induction on $\l$ it suffices to prove the lemma for $\l=1$, in which case the last trapezoid is simply $\widehat T_i=\widehat{ST}_{i,1}$. 

We will proceed in two steps. First, we will divide $\widehat V$ into $\widetilde V$ and $\widehat{ST}_{i,1}$. The $\widetilde V$ part is harmless, as it directly relates to $\g_{\widetilde V}$ appearing in the r.h.s.\ of the statement of the lemma. In order to reproduce a `resampling' of $\widehat{ST}_{i,1}$ we will proceed in two steps. First, using the FA1f-like dynamics, Lemma \ref{lem:FAperiodic}, we will reduce the problem to resampling a single site in $\widehat{ST}_{i,1}$ given that next to it there are $w$ consecutive infections. Then we will use $\widetilde V$ to provide additional infections to ensure that $c_x^{\widehat V}$ is satisfied and this will yield the $x$ term of the Dirichlet form from \eqref{eq:Dirichlet}. The lemma is illustrated in Figure \ref{fig:slice}.

Recalling Remark \ref{rem:prod}, for any $f:\O_{\widehat V}\to \bbR$ Lemma~\ref{lem:var:conv} gives
  \begin{equation}
   \label{eq:9bis}
    \var_{\widehat V}(f\tc \SG(\widehat V)) \le \mu_{\widehat V}\left(\var_{\widetilde V}(f\tc \SG(\widetilde V)) +
   \var_{\widehat{ST}_{i,1}}(f\tc \widehat \ST_{i,1})\tc \SG(\widehat V)\right). 
  \end{equation}
Since $\SG(\widehat V)=\SG(\widetilde V)\times\widehat{ST}_{i,1}$, the first term in the r.h.s.\ above is
\begin{align*}
\mu_{\widehat{ST}_{i,1}}\left(\var_{\widetilde V}(f\tc \SG(\widetilde V))\tc \widehat \ST_{i,1}\right)&\le \frac{\g_{\widetilde V}}{\mu_{\widehat{ST}_{i,1}}(\widehat\ST_{i,1})}\mu_{\widehat V}\left(\sum_{x\in\widetilde V}c^{\widehat  V}_x\var_x(f)\right)\\
&=(1+o(1))\g_{\widetilde V}\sum_{x\in\widetilde V}\m_{\widehat V}(c_x^{\widehat V}\var_x(f))
\end{align*}
by the definition \eqref{eq:3} of $\g_{\widetilde V}$, Observation \ref{obs:prob:G}, and the fact that $c^{\widetilde V}_x\le c_x^{\widehat V}$. To bound the second term in \eqref{eq:9bis}, we use Lemma \ref{lem:FAperiodic} for $\mu_{\widehat{ST}_{i,1}}(\cdot\tc \widehat \ST_{i,1}) $ with $\k=w$ and constraining event $\cH=\{0\}\subset \{0,1\}=\widehat S$, the hypothesis of the lemma following from Observation \ref{obs:prob:G}. This gives
\begin{equation}
\label{eq:100}  \var_{\widehat{ST}_{i,1}}(f\tc \widehat \ST_{i,1})\le 
q^{-O(w)}\sum_{x\in \widehat{ST}_{i,1}}\mu_{\widehat{ST}_{i,1}}(\1_{\cH_x}\var_x(f)),
\end{equation}
where $\cH_x$ is the event that $w$ consecutive sites immediately to the left or to the right of $x$ in $\widehat{ST}_{i,1}$ are infected. Plugging this back in \eqref{eq:9bis}, we see that we need to bound from above a generic term 
\begin{equation}
\label{eq:generic:term}\mu_{\widehat V}\left(\1_{\cH_x}\var_x(f)\tc \SG(\widetilde V)\right),\quad x\in \widehat{ST}_{i,1}.
\end{equation}

At this point we have succeeded in bringing $w$ consecutive infected sites next to the site $x$, which we want to update. In order to be sure that the constraint $c_x^{\widehat V}$ is satisfied, we would like to also bring some infections next to $x$ in $\widetilde V$. To do that we first use Lemma \ref{lem:var:conv} to include $\widetilde V$ in the variance, so that we are allowed to `resample' it and then use the two-block dynamics, Lemma \ref{lem:2block}, to indeed obtain the desired infections by resampling $\widetilde V$ enough times.

Applying Lemma~\ref{lem:var:conv} to $\L_1=\widetilde V$, $\L_2=\{x\}$, $\L_3=\varnothing$, $\nu_1=\mu_{\widetilde V}(\cdot\tc \SG(\widetilde V))$, and $ \n_2=\mu_x$, we bound the generic term \eqref{eq:generic:term} from above by 
\[
  \mu_{\widehat{ST}_{i,1}\setminus\{x\}}\left(\1_{\cH_x}\var_{\widetilde V\cup \{x\}}(f\tc
  \SG(\widetilde V))\right).
\]
We next apply Lemma \ref{lem:2block} to the product space
\[(\O_{\widetilde V},\mu_{\widetilde V}(\cdot\tc G_{\widetilde V}))\otimes(\O_{\{x\}},\mu_x)\]
with constraining event $\widetilde\cH_x\subset \O_{\widetilde V}$ that the trapezoid 
\[\L_x=x+\left(\overline \bbH_{u_{i+k-1}}(w^2)\cap\bbH_{u_{i+k}}\cap\overline \bbH_{u_{i+k+1}}(w^2)\cap\bbH_{u_{i+3k}}(w)\right)\cap\widetilde V\]
being infected. It is not hard to check that $\partial_{u_{i+k}}\L_x$ contains $x$ and the $w$ infected sites guaranteed by $\cH_x$. In other words, we are in the setting of Figure~\ref{fig:slice}. Using $|\L_x\cap \bbZ^2|=O(w^4)$ and noticing that by the Harris inequality~\cite{Harris60} $\mu_{\widetilde V}(\widetilde\cH_x\tc \SG(\widetilde V))\ge \mu_{\widetilde V}(\widetilde\cH_x),$ Lemma \ref{lem:2block} gives 
\begin{multline*}
  \1_{\cH_x}\var_{\widetilde V\cup \{x\}}(f\tc \SG(\widetilde V))\le
    q^{-O(w^4)} \mu_{\widetilde V\cup\{x\}}\big(\var_{\widetilde V}(f\tc
  \SG(\widetilde V))\\
  +\1_{\cH_x\cap \widetilde\cH_x}\var_x(f)\tc \SG(\widetilde V)\big).
\end{multline*}
Finally, since $u_{i+k}$ is an isolated (quasi-)stable direction, it is easily seen (see Figure \ref{fig:slice} and Lemma \ref{lem:strip}) that $\1_{\cH_x\cap \widetilde\cH_x}\le c^{\widehat V}_x$. Recalling the definition of the Poincar\'e constant $\g_{\widehat V}$ (see \eqref{eq:3}), we conclude that
\begin{multline*}
\mu_{\widehat{ST}_{i,1}\setminus\{x\}}\left(\1_{\cH_x}\var_{\widetilde V\cup \{x\}}(f\tc\SG(\widetilde V))\right)\\
\le q^{-O(w^4)} \left(\g_{\widetilde V}+1/\mu(\SG(\widetilde V))\right)\sum_{y\in \widetilde V\cup \{x\}}\mu_{\widehat V}(c^{\widehat V}_y\var_y(f)).
\end{multline*}

Putting all together, we finally get 
\begin{multline*}
\var_{\widehat V}(f\tc \SG(\widehat V))\le |\widehat{ST}_{i,1}|q^{-O(w^4)}
\max\left(\g_{\widetilde V},1/\mu_{\widetilde V}(\SG(\widetilde V))\right)\\
\times\sum_{x\in \widehat V}\mu_{\widehat V}\left(c^{\widehat V}_x \var_x(f)\right),
\end{multline*}
where the factor $|\widehat{ST}_{i,1}|=O(\widehat L)=O(q^{-4w})$ comes from the fact that each vertex $x\in \widehat{ST}_{i,1}$
produces a term of the form $\sum_{y\in \widetilde V}\mu_{\widehat V}(c^{\widehat V}_y \var_y(f))$. \end{proof}

The remaining induction step allows us to reduce the size of the last non-empty trapezoid $\widehat T_i$ twice.
The proof is illustrated in Figure~\ref{fig:shift}.

\begin{lem}[Bisection of a trapezoid]
\label{prop:bisection}
Let $\widehat V=V^R_{\widehat L}(\widehat \ur)$ be a snail of type $i$ such that $\widehat r_i$ is larger than some sufficiently large constant. Let $\l=\min\{\ell>0,\ell u_{i+1}\in\bbZ^2\}=O(1)$ and let $x=u_{i+1}\l\lfloor \widehat r_i/(2\l\<u_{i+k},u_{i+1}\>)\rfloor$. With this choice $\<u_{i+k},x\>\simeq \widehat r_i/2$. In other words, $x$ is the vector by which the ring should be translated so that half of the last trapezoid, $\widehat  T_i$, remains above it (see Figure \ref{fig:shift}). Then we set:
\begin{align*}
\widetilde \ur&{}=(\widehat r_0,\dots,\widehat r_{i-1},\<u_{i+k},x\>,0,\dots,0),\\
\overline \ur&{}=(\widehat r_0-\<u_k,x\>,\dots,\widehat r_{i}-\<u_{k+i},x\>,0,\dots,0),\\
\overline L&{}=\min\left(\widehat L,\widehat L-\frac{\<u_{k-1},x\>}{\<u_{k-1},u_0\>}\right),\\
\widetilde V&{}=V^R_{\widehat L}(\widetilde \ur),\\
\overline V&{}=x+V^R_{\overline L}(\overline \ur).
\end{align*}
In words, $\widetilde V$ is $\widehat V$ with half of $\widehat T_i$ removed, while $\overline V$ is the snail such that its last trapezoid $\overline T_i$ is exactly that missing half, but with length eventually shortened, so that $\overline V$ fits inside $\widehat V$ (see Figure \ref{fig:shift}). With these notations,
\[\g_{\widehat V}\le\g_{\widetilde V}/\mu(G_{\overline V})+\g_{\overline V}/\mu(G_{\widetilde V})\]
and $\widetilde V$ and $\overline V$ are snails of type $i$.
\end{lem}

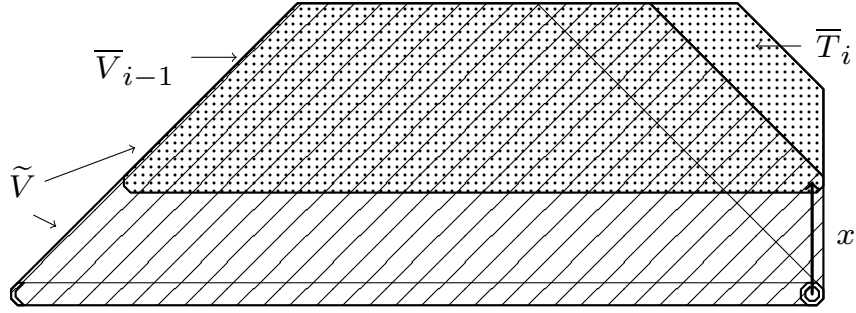
\begin{figure}
\begin{tikzpicture}[line cap=round,line join=round,x=0.15cm,y=0.15cm, scale=1]
\clip(-72,-2) rectangle (4,26);
\draw [thick] (-0.41,1) -- (0.41,1) -- (1,0.41) -- (1,-0.41) -- (0.41,-1) -- (-0.41,-1) -- (-1,-0.41) -- (-1,0.41) -- cycle;
\draw [thick] (-0.25,0.6) -- (0.25,0.6) -- (0.6,0.25) -- (0.6,-0.25) -- (0.25,-0.6) -- (-0.25,-0.6) -- (-0.6,-0.25) -- (-0.6,0.25) -- cycle;
\draw [thick] (-69.85,1) -- (-70.41,1) -- (-71,0.41) -- (-71,-0.41) -- (-70.41,-1) -- (-69.85,-1) -- (-70.6,-0.25) -- (-70.6,0.25) -- cycle;
\draw[thick,pattern=dots] (-45.62,25.79) -- (-6.61,25.79) -- (1,18.18) -- (1,9.59) -- (0.41,9) -- (-60.41,9) -- (-61,9.59) -- (-61,10.41) -- cycle;
\draw[thick, pattern=my north east lines] (1,-0.41) -- (0.41,-1) -- (-70.41,-1) -- (-71,-0.41) -- (-71,0.41) -- (-45.62,25.79) -- (-14.38,25.79) -- (1,10.41) -- cycle;
\draw[thick, opacity=1] (-14.38,25.79) -- (-6.61,25.79) -- (1,18.18) -- (1,10.41) -- cycle;
\draw (-0.41,1)-- (0.41,1);
\draw (0.41,1)-- (1,0.41);
\draw (1,0.41)-- (1,-0.41);
\draw (1,-0.41)-- (0.41,-1);
\draw (0.41,-1)-- (-0.41,-1);
\draw (-0.41,-1)-- (-1,-0.41);
\draw (-1,-0.41)-- (-1,0.41);
\draw (-1,0.41)-- (-0.41,1);
\draw (-0.25,0.6)-- (0.25,0.6);
\draw (0.25,0.6)-- (0.6,0.25);
\draw (0.6,0.25)-- (0.6,-0.25);
\draw (0.6,-0.25)-- (0.25,-0.6);
\draw (0.25,-0.6)-- (-0.25,-0.6);
\draw (-0.25,-0.6)-- (-0.6,-0.25);
\draw (-0.6,-0.25)-- (-0.6,0.25);
\draw (-0.6,0.25)-- (-0.25,0.6);
\draw (-0.6,0.25)-- (-0.6,-0.25);
\draw (-0.6,-0.25)-- (-0.25,-0.6);
\draw (-0.25,-0.6)-- (0.25,-0.6);
\draw (0.25,-0.6)-- (0.6,-0.25);
\draw (0.6,-0.25)-- (0.6,0.25);
\draw (0.6,0.25)-- (0.25,0.6);
\draw (0.25,0.6)-- (-0.25,0.6);
\draw (-0.25,0.6)-- (-0.6,0.25);
\draw (-70.41,1)-- (-71,0.41);
\draw (-71,0.41)-- (-71,-0.41);
\draw (-70.41,-1)-- (-71,-0.41);
\draw (-70.41,1)-- (-45.62,25.79);
\draw (-45.62,25.79)-- (-24.38,25.79);
\draw (-24.38,25.79)-- (0.41,1);
\draw (0.41,1)-- (-70.41,1);
\draw (-70.41,-1)-- (-0.41,-1);
\draw [->, very thick] (0,0) -- (0,10);
\draw (-69.85,1)-- (-70.41,1);
\draw (-70.41,1)-- (-71,0.41);
\draw (-71,0.41)-- (-71,-0.41);
\draw (-71,-0.41)-- (-70.41,-1);
\draw (-70.41,-1)-- (-69.85,-1);
\draw (-69.85,-1)-- (-70.6,-0.25);
\draw (-70.6,-0.25)-- (-70.6,0.25);
\draw (-70.6,0.25)-- (-69.85,1);
\draw (-45.62,25.79)-- (-6.61,25.79);
\draw (-6.61,25.79)-- (1,18.18);
\draw (1,18.18)-- (1,9.59);
\draw (1,9.59)-- (0.41,9);
\draw (0.41,9)-- (-60.41,9);
\draw (-60.41,9)-- (-61,9.59);
\draw (-61,9.59)-- (-61,10.41);
\draw (-61,10.41)-- (-45.62,25.79);
\draw (1,-0.41)-- (0.41,-1);
\draw (0.41,-1)-- (-70.41,-1);
\draw (-70.41,-1)-- (-71,-0.41);
\draw (-71,-0.41)-- (-71,0.41);
\draw (-71,0.41)-- (-45.62,25.79);
\draw (-45.62,25.79)-- (-14.38,25.79);
\draw (-14.38,25.79)-- (1,10.41);
\draw (1,10.41)-- (1,-0.41);
\draw (-14.38,25.79)-- (-6.61,25.79);
\draw (-6.61,25.79)-- (1,18.18);
\draw (1,18.18)-- (1,10.41);
\draw (1,10.41)-- (-14.38,25.79);
\begin{scriptsize}
\draw (3,5) node [scale=1.5] {$x$};
\draw (-70,10) node [scale=1.5] {$\widetilde V$};
\draw (-60,20) node [scale=1.5] {$\overline V_{i-1}$};
\draw (2,22) node [scale=1.5] {$\overline T_i$};
\end{scriptsize}
\draw [->] (-69,7)--(-67,6);
\draw [->] (-68,10)--(-60,13);
\draw [->] (-55,21)--(-51,21);
\draw [->] (-1,22)--(-5,22);
\end{tikzpicture}
\caption{
The geometric setting of Lemma~\ref{prop:bisection}. The snail $\widetilde V$ is hatched, $\overline V$ is dotted and their union is the original snail $\widetilde V$. The dotted-hatched region is $\overline V_{i-1}$, while the dotted trapezoid is $\overline T_i$.} 

\label{fig:shift}
\end{figure}

\begin{proof}[Proof of Lemma \ref{prop:bisection}]
The proof goes as follows. In Claim \ref{claim:shift} we show that the two polygons $\widetilde V$ and $\overline V$ are indeed snails (defined by admissible sequences) of type $i$ and that they do correspond to their informal definitions in the statement of the lemma. Though technical, this claim hides no subtlety and we invite the reader to skip it. Then we apply Lemma \ref{lem:2block} to reduce the problem of relaxation on $\widehat V$ to the one on $\widetilde V$ and on $\overline V$ which yields the desired result. The event $\SG(\widetilde V)$ is implied by $\SG(\widehat V)$ by construction, but the second block, $\overline V$, of the dynamics corresponding to Lemma \ref{lem:2block}, is updated only when the part of $\SG(\overline V)$ witnessed in $\widetilde V$ occurs.

We begin with some geometric observations following directly from Definitions~\ref{def:snail} and \ref{def:ti}.
\begin{claim}
\label{claim:shift}
$\widetilde V$ and $\overline V$ are snails of type $i$. Furthermore, we have $\widetilde V\cup\overline V=\widehat V$ and $\widehat V\setminus\widetilde V=\overline T_i$ (the last trapezoid of $\overline V$).
\end{claim}
\begin{proof}[Proof of the claim]
The statement that $\widetilde V$ is a snail of type $i$ follows from the definition of $\widetilde \ur$ and the same fact for $\widehat V$, since $\widetilde r_i\le \widehat r_i$. 

Turning to $\overline V$, notice that $\<u_j,x\>\ge 0$ for all $i\in[2k]$ and $j\in[k,k+i]$ with equality iff $i=2k-1$ and $j=k$. Thus, for all $j\in[2k]$ we have $ \widehat r_j\ge\overline r_j$ and clearly $\widehat L\ge \overline L$. Thus, recalling the definition of $\overline V$ and that $\widehat V$ is of type $i$, conditions \ref{item:type:i} and \ref{item:ri:smaller} and the left inequalities in \ref{item:r:difference} and \ref{item:L:difference} of Definition \ref{def:relevant} are satisfied. Moreover, for $0\le j<i$ we have 
\begin{equation}
\label{eq:r:diff}\widehat r_j-\overline r_j=\<u_{k+j},x\>=\frac{\<u_{i+1},u_{k+j}\>}{\<u_{i+1},u_{k+i}\>}(\widehat r_i-\overline r_i)\le C(\widehat r_i-\overline r_i),
\end{equation}
so the right inequality of \ref{item:r:difference} for $\overline V$ follows from the one for $\widehat V$. Similarly, 
\begin{equation}
    \label{eq:L:diff}
    \widehat L-\overline L\le \frac{|\<u_{k-1},x\>|}{\<u_{k-1},u_0\>}\le C(\widehat r_i-\overline r_i),
\end{equation}
gives that \ref{item:L:difference} of Definition \ref{def:relevant} holds for $\overline V$.

We next prove that $\overline V$ is a snail (with admissible $\overline L$ and $\overline\ur$). Recalling from Definition \ref{def:snail} that we need to prove that 
\begin{enumerate}[label=(\roman*),ref=(\roman*)]
    \item\label{item:positive} $\overline r_j\ge 0$ for all $j\in[2k]$,
    \item\label{item:rj:ratio} $\overline r_j\le \d \overline r_{j-1}$ for all $1\le j<2k$,
    \item\label{item:r0:ratio} $\overline r_0\le \d \overline L$, and
    \item\label{item:rational:line} $(\widehat L-\overline L)\<u_0,u_{k-1}\>/\r_{k-1}\in\bbN$.
\end{enumerate}

To check \ref{item:positive}, observe that $\overline r_j=\widehat r_j-O(\widehat r_i)>0$ for $j<i$ by admissibility of $\widehat \ur$ and $\overline r_i\simeq \widetilde r_i/2>0$. By admissibility of $\widehat V$, $\widehat r_j-\overline r_j=\<u_{k+j},x\>=\Theta(\widehat r_i)$ and $C<1/\d$ we get \ref{item:rj:ratio}. For the last two properties we consider two cases. 

First assume that $i\in\{2k-2,2k-1\}$ (i.e.\ $x$ corresponds to a horizontal translation to the right---in direction $u_{2k}$). It is easy to check from the definition of $\overline L$ that in this case $\overline L=\widehat L$, so that \ref{item:rational:line} is trivial and \ref{item:r0:ratio} follows from $\overline r_0\le \widehat r_0\le \d \widehat L$. This concludes the proof that $\overline V$ is a snail of type $i$ in this case.

Assume that, on the contrary, $i<2k-2$, so that the $\widehat L-\overline L=\frac{\<u_{k-1},x\>}{\<u_{k-1},u_0\>}=\Theta(\widehat r_i)$. Then  \ref{item:r0:ratio} follows from the fact that $\widehat r_0-\overline r_0=\Theta(\widehat r_i)$ as above. For \ref{item:rational:line} simply observe that
\[(\widehat L-\overline L)\<u_0,u_{k-1}\>=\<u_{k-1},x\>\in\r_{k-1}\bbN,\]
since $x\in\bbZ^2$ by the definition of $x$ and $\l$. This concludes the proof that $\overline V$ is a snail of type $i$.

By Definition~\ref{def:ti} it is clear that
\begin{multline*}
\widehat V\setminus\overline V=\left(\overline\bbH_{u_{k+i}}(R_{k+i}+\widehat r_i)\setminus\overline\bbH_{u_{k+i}}(R_{k+i}+\<u_{k+i},x\>)\right)\\
\cap\overline\bbH_{u_{k+i-1}}(R_{k+i-1}+\widehat r_{i-1})\cap\overline\bbH_{u_{k+i+1}}(R_{k+i+1}).
\end{multline*}
It also follows from Definition~\ref{def:ti} that the above trapezoid $\widehat V\setminus \overline V$ is also equal to $\overline T_i$ as claimed. Finally, we have that $\overline V\subset \widehat V$ using Definition~\ref{def:snail}, which completes the proof of the claim.
\end{proof}
Let now
\[
\overline V_{i-1}=V^R_{\overline L}(\overline r_0,\dots,\overline r_{i-1},0,\dots,0)=\overline V\setminus\overline T_i=\overline V\cap\widetilde V.
\]
By Claim~\ref{claim:shift} and Remark \ref{rem:prod} we have
\begin{equation}
\label{eq:OV,muV}
(\O_V,\mu_{\widehat V}(\cdot\tc G_{\widehat V}))= (\O_{\widetilde V}, \mu_{\widetilde V}(\cdot\tc \SG(\widetilde V)))\otimes(\O_{\overline T_i}, \mu_{\overline T_i}(\cdot\tc \cG(\overline T_i)))
\end{equation}
and we can apply Lemma \ref{lem:2block} with the facilitating event 
\[\SG(\overline V_{i-1})=\SG(\overline B)\cap\bigcap_{j<i}\cG(\overline T_j)\subset \O_{\widetilde V},\]
where $\overline B$ and $\overline T_j$ are the base and trapezoids of $\overline V$. We get 
\begin{equation}
\label{eq:rec3}
\begin{multlined}
 \var_{\widehat V}(f\tc \SG(\widehat V)) 
\le \mu(\SG(\overline V_{i-1}))^{-1}\mu_{\widehat V}\big(\var_{\widetilde V}(f\tc \SG(\widetilde V))\\ +\1_{\SG(\overline V_{i-1})}\var_{\overline T_{i}}(f\tc \cG(\overline T_i))\tc \SG(\widehat V)\big),
\end{multlined}
\end{equation}
where we used that $\mu_{\widetilde V}(\SG(\overline V_{i-1})\tc \SG(\widetilde V))\ge\mu(\SG(\overline V_{i-1}))$ by the Harris inequality.
Using the definition of the Poincar\'e constant $\g_{\widetilde V},$ the fact that $c_x^{\widetilde V}\le c_x^V$ together with $\mu(\SG(\overline V_{i-1}))\ge \mu(\SG(\overline V))$ the first term is bounded from above by
\begin{equation}
  \label{eq:1term}
\frac{\g_{\widetilde V}}{\mu(\SG(\overline V))} \mu_{\widehat V}\left(\sum_{x\in \widetilde V}c^{\widehat V}_x\var_x(f)\right).
\end{equation}
The term $\mu(\SG(\overline V_{i-1}))^{-1}\mu_{\widehat V}\left({\1}_{\SG(\overline V_{i-1})}\var_{\overline T_i}(f\tc \cG(\overline T_i))\tc \SG(\widehat V)\right)$ from the r.h.s.\ of \eqref{eq:rec3} can be bounded from above by 
\begin{multline}
\label{eq:2term}
\mu(\SG(\widetilde V))^{-1}\mu_{\widetilde V}\left(\mu_{\overline V_{i-1}}(\var_{\overline T_i}(f\tc \cG(\overline T_i))\tc \SG(\overline V_{i-1}))\right)\\
\begin{aligned}[t]\le{}&\mu(\SG(\widetilde V))^{-1}\mu_{\widetilde V}\left(\var_{\overline V}(f\tc \SG(\overline V))\right)\\
\le{}& \frac{\g_{\overline V}}{\mu(\cG(\widehat V))}\mu_{\widetilde V}\left(\sum_{x\in \overline V}c^{\widehat V}_x\var_x(f)\right),
\end{aligned}
\end{multline}
using \eqref{eq:OV,muV}, Lemma~\ref{lem:var:conv}, the definition of $\g_{\overline V}$ and the fact that $c_x^{\overline V}\le c_x^{\widehat V}$.

If we now combine \eqref{eq:rec3}, \eqref{eq:1term} and
\eqref{eq:2term} we get the statement of the lemma.
\end{proof}

We can now assemble our main induction step from Lemmas~\ref{prop:slice} and \ref{prop:bisection}. Namely, we repeatedly use Lemma \ref{prop:bisection} until the last trapezoid is reduced to a bounded number of lines and then apply Lemma \ref{prop:slice} to remove them as well.
\begin{cor}[Removing a trapezoid]
\label{cor:trapezoid}
Let $\s_i=1/\min_{V_i}\mu(\SG(V^i))$ with $\min$ running over all snails of type $i$. Let $\G_i=\max_{V_i'}\g_{V_i'}$, where the $\max$ runs over all snails of type $i$ with $r_i=0$. Let $\widehat V=V^R_{\widehat L}(\widehat \ur)$ be a snail of type $i$. Then
\[\g_{\widehat V}\le q^{-O(w^4)}\s_i^{O(\max(1,\log \widehat r_i))}\G_i.\]
\end{cor}
\begin{proof}[Proof of Corollary \ref{cor:trapezoid}]
Let $c$ be a sufficiently large constant. We prove by induction on $\widehat r_i$ that
\[\g_{\widehat V}\le q^{-cw^4}\s_i^{c\max(1,\log \widehat r_i)}\G_i.\]

The base of the induction, $\widehat r_i\le\sqrt{c}$, follows from Lemma \ref{prop:slice}, since $\g_{\widetilde V}\ge 1$ by definition. Assume that $\widehat r_i>\sqrt{c}$. Then Lemma \ref{prop:bisection} and the induction hypothesis applied to both $\widetilde V$ and $\overline V$ from that lemma give
\[\g_{\widehat V}\le \s_iq^{-cw^4}\s_i^{c\log (2\widehat r_i/3)}\G_i\le q^{-cw^4}\s_i^{c\log \widehat r_i}\G_i,\]
since both $\widetilde r_i$ and $\overline r_i$ in Lemma \ref{prop:bisection} are smaller than $2\widehat r_i/3$. This completes the proof of the induction step and the corollary.
\end{proof}
We are now ready to conclude the proof of Proposition~\ref{prop:snail:reduction} and of Theorem \ref{thm:key step}. 
\begin{proof}[Proof of Proposition~\ref{prop:snail:reduction}]
Applying Corollary \ref{cor:trapezoid} to each non-zero coordinate of $\ur$, we obtain
\[\g_{V}\le\left(q^{-O(w^4)}\s\right)^{O(\log L)}\G\]
with the notation of the statement of Proposition~\ref{prop:snail:reduction}.
\end{proof}
\begin{proof}[Proof of Theorem \ref{thm:key step}]
Combining Propositions \ref{prop:base case} and \ref{prop:snail:reduction} we get
\[\g_{V}\le \left(q^{-w^4}\s q^{-Rw}\right)^{O(\log L)}\le e^{-O(w^4\log^3(1/q)/q^\a)},\]
where the last equality follows from Observation \ref{obs:prob:SG}.
\end{proof}

\section{Proof of Theorem~\ref{th:main}}
\label{sec:Proof}
Recall that $w$ is a large constant much
bigger than the constants in any $O(\cdot)$ notation.
Let $t_*= \frac 1w e^{w^5\log^3(1/q)/q^\a}$ and $T=e^{1/q^{3\a}}$. Then we have
\begin{align*}
  \bbE_\mu(\t_0)={}&\int_0^{+\infty}
                     \ds\bbP_\mu(\t_0>s)\\
                     ={}&\int_{0}^{t_*}
                     \ds\bbP_\mu(\t_0>s) + \int_{t_*}^T \ds\bbP_\mu(\t_0>s)
  +\int_T^{+\infty} \ds \bbP_\mu(\t_0>s) \\
  \le{}& t_* + T\bbP_\mu(\t_0>t_*)
  +\int_T^{+\infty} \ds \bbP_\mu(\t_0>s).  
\end{align*}
The term $t_*$ has exactly the form required in
Theorem~\ref{th:main}. The last term in the r.h.s.\ above tends to zero as $q\to
0$. Indeed, using \cite{Martinelli19a}*{Theorem 2} we have that
$\forall s>0,\bbP_\mu(\t_0>s)\le e^{-s \l_0}$
with $\l_0 \ge e^{-\O((\log q)^4/q^{2\a})}$ and therefore
\[
\int_T^{+\infty} \ds \bbP_\mu(\t_0>s)\le \l_0^{-1}e^{-T\l_0}\to
0\quad\text{ as $q\to 0$}.
\]
In conclusion, the proof of the upper bound in Theorem~\ref{th:main}
boils down to proving
\begin{equation}
  \label{eq:caldo}
  \lim_{q\to 0}T\bbP_\mu(\t_0>t_*)=0.  
\end{equation}
That
requires a sequence of simple steps (\ref{step:a}-\ref{step:d} below) and a more
involved part (\ref{step:e} below).
Before turning to the details of the proof of Theorem~\ref{th:main},
let us sketch our approach.
\subsection{Roadmap}
\begin{enumerate}[label=(\alph*),ref=(\alph*)]
\item\label{step:a} In order to prove that w.h.p.\ $\t_0\le t_*$, it suffices to prove the result
for the (stationary) $\cU$-KCM process on to the torus $\L$ and with side $K=2e^{w^5\log^3(1/q)/q^\a}$ (see \eqref{eq:5}).
\item\label{step:b} Let $L=\Theta(\l)/q^{3w}$ for a large positive constant $\l=\l(\cU,\d)$, let
$R=w^2\log(1/q)/q^\a$, and recall the
good and super good events described in Section \ref{sec:good-supergood} and Definition \ref{def:good events}. Given a snail $V=V^R_L(\ur)=B\cup \bigcup_{i\in[2k]} T_i^\pm\subset \L$ (recall Definitions \ref{def:snail} and \ref{def:ti}) with base $B$ and trapezoids $T_i^\pm$, we will construct  a new event
$\cE\subset \O_{\L\cap \bbZ^2}$ which will guarantee that (in
particular) the following occurs. 
\begin{enumerate}[label=(\roman*),ref=(\alph{enumi}.\roman*)]
\item\label{item:i} For any (translate of)  $V\subset \L$ as above, the good events $\cG(T_i^\pm)$ occur for all $i\in [2k]$.
\item\label{item:ii} In every strip of $\L$ parallel to $u_0$ and of width $2R$ there exists a translate of the base $B$ for which the super good event holds.
\end{enumerate}
\item\label{step:c} We will prove that $\mu(\cE)\ge 1- e^{-1/q^{w}}$, which will allow us to conclude that it is sufficient to analyse the infection time of the
origin of the stationary  $\cU$-KCM
in $\L$ restricted to $\cE$ (see \eqref{eq:10}).
\item\label{step:d} For the latter process we will follow the
standard ``variational'' approach (see \cite{Asselah01}*{Theorem 2} and
also \cite{Martinelli19a}*{Section 2.2}) and get that
\[
  T\bbP_\mu(\t_0\ge t_*)\le Te^{-t_* \l_\cF}(1 +o(1)).
  \]
Here $\l_\cF$ is related to the Dirichlet problem for the $\cU$-KCM on the torus and restricted to $\cE$ with boundary condition $f\big|_{\{\o\in \cE:\o_0=0\}}=0$. In particular (see \eqref{eq:casa}) \[
\l_\cF\ge\inf_f q\frac{\cD^{\text{per}}_\L(f)}{\var_\L(f\tc \cE)},
\]
where $\cD^{\text{per}}_\L(f)$ is 
the Dirichlet form of the $\cU$-KCM on the torus $\L$ and the supremum is taken over all $f:\cE\to \bbR$.
\item\label{step:e} The last and most important step will be to prove that
\begin{equation*}
  \var_\L(f\tc \cE)\le e^{O(w^4\log^3(1/q))/q^\a} \cD^{\text{per}}_\L(f),
\end{equation*}
implying that $t_*\l_\cF$ diverges as $q\to 0$ rapidly enough. The high-level intuition behind the above Poincar\'e inequality is as
follows. A super good base (i.e.\ a base $B$ for which the super good $\cS\cG(B)$ event holds), whose presence is guaranteed by
\ref{item:ii}, will be able to move in $\L$ using an FA1f-like dynamics as in Lemma \ref{lem:FAperiodic} with $\widehat\nu(\cH)$ given by
\[
e^{\Theta(w^3\log^2(1/q))/q^\a}.
\]
Indeed, we will reproduce each step of that dynamics with a resampling of an appropriate super good translate of the snail $V$, since \ref{item:i} guarantees that the super good base does extend to a super good translate of the snail $V$. Indeed, the snail (see Figure \ref{fig:snails}) does extend on both sides of the base for a distance $\Theta(r_{2k-1})$, so taking $r_{2k-1}$ of order $L$, one can induce a change on both sides of the base by resampling the configuration inside the snail. Thanks to Theorem \ref{thm:key step}, the relaxation time of the super good snail is $e^{O(w^{4}\log^3(1/q))/q^\a}$. The conclusion of Theorem~\ref{th:main} then follows rather naturally.
\end{enumerate}

\subsection{Proof}
Let $K=2\exp(w^5\log^3 (1/q)/q^{\a})$ and let 
$\L=\bbR^2/(Ku_0\bbZ+Ku_k\bbZ)$
be the torus in $\bbR^2$ of side $K$ directed by $u_0$, which we think of as centred at $0$. Further set
\begin{align*}
R=&w^2\log(1/q)/q^\a, &W=&1/q^{3w},& M=&K/(2R_0+W),
\end{align*}
recalling the notation $R_0=\rho_0\lfloor R/\rho_0\rfloor$ from Definition~\ref{def:ring}. For simplicity we assume that $u_0(2R_0+W)\in\bbZ^2$ and that $M$ is an even integer ($W$ and $K$ can be modified by $O(1)$ and $O(1/q^{3w})$ respectively, so that these both hold). 

We partition $\L$ into alternating strips $\L^{(1)}_i,\L_i^{(2)}, i\in [M]$, of length $K$ and parallel to $u_0$ (see Figure \ref{fig:strips}). The strips $\L_i^{(1)}$ have width $2R_0$ while the strips $\L_i^{(2)}$ have width $W$. We write $\L_i=\L_i^{(1)}\cup \L_i^{(2)}$ and we think of the thin strip $\L_i^{(1)}$ as being just below the thick one $\L_i^{(2)}$, when $u_0$ points left. In turn, we partition $\L_i$ into consecutive squares $Q_{i,j}, j\in [M],$ of side length equal to $2R_0+W$ and sides parallel to $u_0$ and $u_k$ and we write $Q_{i,j}^{(a)}=Q_{i,j}\cap \L_i^{(a)}, a\in\{1,2\}$. 
\begin{figure}
    \centering
    \begin{tikzpicture}[scale=0.25]

\foreach \j in {0,...,7}{ 
\draw  [color=gray] (-10,-5+2*\j)--(30,-5+2*\j)--(30,-5+2*\j+2)--(-10,-5+2*\j+2)--cycle;  
\draw  [fill,opacity=0.5] (-10,-5+2*\j)--(30,-5+2*\j)--(30,-5+2*\j+0.5)--(-10,-5+2*\j+0.5)--cycle;  
}
\foreach \j in {0,...,20}{ 
\draw  [color=gray] (-10+2*\j,-5)--(-10+2*\j,11);  
}
\draw [pattern=north east lines, opacity=1] (-6,-3) rectangle (-4,-1);
\begin{scope}
\begin{scriptsize}
\draw [thin,decorate,decoration={brace,amplitude=2pt},xshift=-4pt] (-10,-2.5)--(-10,-1.1) node [black,xshift=-0.5cm] {$\L^{(2)}_i$};
\draw [thin,decorate,decoration={brace,amplitude=0.8pt},xshift=-4pt]
(-10,-5)--(-10,-4.5) node [black,xshift=-0.5cm] {$\L^{(1)}_{i-1}$};
\draw [thick,<->]  (30.3,-3.1)--(30.3,-4.5);
\draw (31,-4) node[anchor=west]
 {$W=\frac{1}{q^{3w}}$};
\draw [thick,<->]  (0,-5.5)--(2,-5.5);
\draw (1,-6.5) node
 {$W+2R_0$};
 \draw [thick,|-|]  (30.3,1)--(30.3,1.5);
\draw (31,1.2) node[anchor=west]
{$2R_0\approx\frac{2w^2}{q^\a} \log\frac{1}{q}$};
\end{scriptsize}
\end{scope}

\draw[line width=1pt] (-6,-3)--(-7,-2)--(-7,3)--(-1,9)--(21,9)--(27,3)--(27,-2)--(26,-3)--cycle;
\draw[line width=1pt] (-3.85,-3)--(-7,0.15);
\draw[line width=1pt] (-4,-2.85)--(-4,6);
\draw[line width=1pt] (-4,-2.65)--(7.65,9);
\draw[line width=1pt] (23.85,-3)--(27,0.15);
\draw[line width=1pt] (24,-2.85)--(24,6);
\draw[line width=1pt] (24,-2.65)--(13.65,9);
\draw[line width=1pt] (-3.85,-2.5)--(23.85,-2.5);
\end{tikzpicture}
\caption{The partition of $\L$ into strips $\L_i=\L_i^{(1)}\cup \L_i^{(2)},i\in [M]$. The hatched region represents a square $Q_{i,j}$, which we would like to resample. The thick polygon is the snail $V$ with its trapezoids. Note that its base $B$ does not intersect $Q_{i,j}$ and (almost) spans the squares $Q_{i,j+1},\dots,Q_{i,j+\l}$.}
    \label{fig:strips}
  \end{figure}
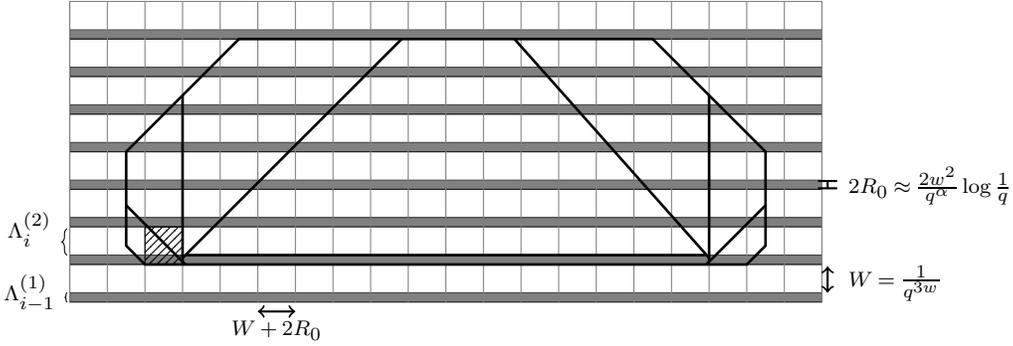
\begin{rem}
Recalling Definition~\ref{def:ring}, the width of the thin strips is chosen so that an annulus $A$ of radius $R$ would fit tightly inside.
\end{rem}
We are now ready to detail the steps \ref{step:a}-\ref{step:e} sketched in the roadmap above.

\subsubsection{Step \ref{step:a}} Notice that $t_*= K/(2w)$ and let $\t_0,\t_0^\L$
denote the infection times of the origin for the $\cU$-KCM process on
$\bbZ^2$ and for the $\cU$-KCM process on the discrete torus $\L\cap \bbZ^2$ respectively. Using the fact that the jump rates of the KCM are bounded, a standard argument of finite speed of information propagation (see e.g.\ \cite{Liggett05}) implies that
\begin{equation}
  \label{eq:5}
  \bbP_\mu\left(\t_0\ge t_*\right) \le \bbP_{\mu_\L}\left(\t^\L_0\ge t_*\right) + e^{-\O(K)}\quad \text{as }q\to 0.
\end{equation}

\subsubsection{Step \ref{step:b}} Given a small positive constant $\varepsilon=\varepsilon(\cU)=\O(1)$ and a large one $\l=\l(\cU,\d)$ to be specified later (recall the constant $\d$ from Definition \ref{def:snail}) let
\begin{enumerate}[label=(\roman*),ref=(\roman*)]
\item\label{event:i} $\SG\left(Q^{(1)}_{i,j}\right)$ be  the event that the rightmost and leftmost annuli $A$ in $Q^{(1)}_{i,j}$ are infected and any segment $I\subset Q^{(1)}_{i,j}$ intersecting $\bbZ^2$, of length $\varepsilon R$ and
  orthogonal to some $u_i\in \hS_0$ contains an infected $u_i$-helping set in $Q_{i,j}^{(1)}$;
\item\label{event:ii} $\cG(Q_{i,j})$ be the event that any segment $I\subset Q_{i,j}$ intersecting $\bbZ^2$, of length $\varepsilon W$ and orthogonal to some $u\in \hS$ contains $w$ infected consecutive sites;
\item\label{event:iii} $\cE_{i}$ be the event that for all the squares $Q_{i,j}\subset
    \L_i$ the event $\cG(Q_{i,j})$ holds and moreover there exists $j\in
    [M]$ such that $\bigcap_{j'=j+1}^{j+\l}\SG\left(Q^{(1)}_{i,j'}\right)$ also holds;
   \item\label{event:iv} $\cE=\bigcap_{i\in[M]}\cE_i$.
\end{enumerate}
\begin{rem}
Similarly to Definition \ref{def:good events}, \ref{event:i} needs to be modified slightly if Assumption \ref{ass:1} is not satisfied, but we keep working under that assumption.
\end{rem}

\subsubsection{Step \ref{step:c}} With our choice of $K,R,W$, as in Observation \ref{obs:prob:SG}, it follows that
$\mu\left(\SG\left(Q^{(1)}_{i,j}\right)\right)=q^{O(Rw)}$. Moreover, using the Harris inequality 
\begin{equation}
\label{eq:muAijRij}
\mu\left(\bigcap_{j\in[\l]}\SG\left(Q^{(1)}_{i,j}\right)\right)\ge q^{O(\l Rw)}.
\end{equation}
Also,
\[
  \mu\left(\bigcap_{j\in[M]} \cG(Q_{i,j})\right)\ge 1- O\left(MW^2\right)e^{-q^w \varepsilon W/w^2} \geq 1- e^{-q^{-2w+o(1)}}.
\]
In conclusion, 
\begin{equation}
\label{eq:muEi}
1-\mu(\cE_i)\le e^{-q^{-2w+o(1)}}+\left(1-q^{O(\l Rw)}\right)^{\lfloor M/\l\rfloor}\le e^{-q^{-2w+o(1)}}\end{equation}
and $\mu(\cE)\ge 1-M(1-\mu(\cE_i))\ge 1-e^{-q^{-2w+o(1)}}$. Therefore, writing
$\t^\L_{\cE^c}$ for the hitting time of $\cE^c$  for the
$\cU$-KCM process in $\L$ and recalling that $t_*=K/(2w)$, we obtain
\begin{equation}
  \bbP_{\mu_\L}(\t^\L_{\cE^c}\le t_*)\le
O(K^2 t_*)\mu(\cE^c) +e^{-\O(K^2t_*)}\le e^{-q^{-2w+o(1)}}.
\label{eq:7bis}
\end{equation}
In the second inequality above we used a simple union bound over the updates for the $\cU$-KCM in $\L$ together with the fact that the law of the $\cU$-KCM process in $\L$ started from $\mu_\L$ is equal to $\mu_\L$ at any given time and a simple large deviations result on the number of updates.

Thus, if $\cF=\{\o:\o_0=0\}\cup \cE^c$ then
\eqref{eq:5} together with \eqref{eq:7bis} imply that
\begin{equation}
\label{eq:10}
\begin{aligned}
  \bbP_\mu(\t_0\ge t_*)&\le \bbP_\mu(\t^\L_0\ge t_*) + e^{-\O(K)}\\
  &\le \bbP_{\mu_\L}(\t_\cF^\L\ge t_*) +\bbP_{\mu_\L}(\t_{\cE^c}^\L\le
    t_*)+ e^{-\O(K)}\\
&\le \bbP_{\mu_\L}(\t_\cF^\L\ge t_*) + e^{-q^{-2w+o(1)}}.
\end{aligned}
\end{equation}

\subsubsection{Step \ref{step:d}} As in \cite{Asselah01}*{Theorem 2},
\begin{equation}
  \label{eq:14}
  \bbP_{\mu_\L}(\t_\cF^\L\ge t_*)\le e^{-\l_\cF t_*},
\end{equation}
with
\[
  \l_\cF=\inf\left\{\frac{\cD_\L^{\rm per}(f)}{\mu_\L(f^2)},f|_{\cF}=0\right\},
  \]
where $\cD_\L^{\rm per}(f)$ denotes the Dirichlet form of the
$\cU$-KCM process on the torus $\L$ (see \eqref{eq:Dirichlet}). 
Observe now that for any $f:\O_\L\to \bbR$ such that $f|_{\cF}=0$
\begin{align*}
\var_\L(f\tc \cE)={}& \frac 12 \sum_\o \sum_{\o'}\mu_\L(\o\tc \cE)\mu_\L(\o'\tc \cE)
(f(\o)-f(\o'))^2\\
\ge{}& \mu_\L(\o_0=0\tc \cE)\mu_\L(f^2\tc \cE)\ge q\mu_\L(f^2), 
\end{align*}
where for the last inequality we used the Harris inequality
($\{\o:\o_0=0\}$ and $\cE$ are both decreasing events) and the fact that
$f^2\1_{\cE}=f^2$. 
Hence,
\begin{equation}
 \label{eq:casa}\l_\cF \ge q \inf_f \frac{\cD_\L^{\rm per}(f)}{\var_\L(f\tc \cE)}.    
\end{equation}
 Notice the absence of the conditioning event $\cE$ in the Dirichlet form
$\cD_\L^{\rm per}(f)$. 

\subsubsection{Step \ref{step:e}} Our main result on the above variational problem is
as follows. 
\begin{thm}
\label{thm:periodic} For all $w>0$ large enough,
all $\varepsilon>0$ small enough and all $f:\O_\L\to \bbR$    
\begin{equation}
  \label{eq:9}
  \var_\L(f\tc \cE)\le e^{O(w^4(\log (1/q))^3)/q^\a}\cD^{\rm per}_\L(f),
\end{equation}
i.e.
\begin{equation}
\label{eq:lF}
  \l_\cF\ge e^{-O(w^4(\log (1/q))^3)/q^\a}.
\end{equation}
\end{thm}
Before proving this theorem, let us first complete the proof of
\eqref{eq:caldo}.
Using $t_*=w^{-1}\exp(w^5(\log (1/q))^3/q^{\a})$ and \eqref{eq:lF}
we get that for any $w$ large
enough
\[
  t_*\l_\cF\ge 1/q^w\]
which, together with \eqref{eq:10} and \eqref{eq:14} and the choice of
$T=e^{1/q^{3\a}}$, gives
\begin{equation}
  \label{eq:12}
T\bbP_\mu(\t_0\ge t_*)\le T\left(e^{-\l_\cF
    t_*}+ e^{-q^{-2w+o(1)}}\right)\to 0.
\end{equation}
This proves \eqref{eq:caldo} and therefore Theorem \ref{th:main} modulo Theorem \ref{thm:periodic}.

\begin{proof}[Proof of Theorem \ref{thm:periodic}]
The two main ingredients of the proof will be Lemma
\ref{lem:FAperiodic} and Theorem \ref{thm:key step}. The definition of
the event $\cE=\bigcap_i \cE_i$ and the fact that the strips $\L_i$ are disjoint imply that $\mu_\L(\cdot\tc
\cE)=\bigotimes_i \mu_{\L_i}(\cdot\tc \cE_i)$. In turn, 
Lemma \ref{lem:var:conv} gives
\begin{equation}
  \label{eq:200}
  \var_\L(f\tc \cE)\le \sum_i \mu_\L(\var_{\L_i}(f\tc \cE_i)\tc \cE).
\end{equation}
Hence, it is enough to analyse a generic term $\mu_\L(\var_{\L_i}(f\tc
\cE_i)\tc \cE)$ and for this purpose we plan to apply Lemma
\ref{lem:FAperiodic} to bound from above $\var_{\L_i}(f\tc \cE_i)$.

Recall that the strip $\L_i$ is the disjoint union of $M$ squares
$(Q_{i,j})_{j=1}^M$ and recall the definition of the ``single square'' events $\SG\left(Q^{(1)}_{i,j}\right)$ and $\cG(Q_{i,j})$ given in \ref{event:i} and \ref{event:ii} above. Those definitions allow us to write (in what follows the index
$i$ of the strip is fixed)
\[
  \mu_{\L_i}(\cdot\tc \cE_i)=\nu_i\left(\cdot \tc \bigcup_j \bigcap_{j'=j+1}^{j+\l}\SG\left(Q^{(1)}_{i,j'}\right)\right)
\]
where $\nu_{i,j}=\mu_{Q_{i,j}}(\cdot\tc \cG(Q_{i,j}))$ and $\nu_i=\bigotimes_j\nu_{i,j}$. We can now apply Lemma
\ref{lem:FAperiodic} to the product measure $\nu_i$ with
$\SG\left(Q^{(1)}_{i,j}\right)$ as the event $\cH$, $M$ as the parameter $n$, and $\l$ as the parameter $\k$. The choice of the key
parameter $\k$ entering the definition of the associated facilitating events $\cH_{i,j}$ in the periodic case, 
\[
\cH_{i,j}=\bigcap_{j'=j+1}^{j+k}\SG\left(Q^{(1)}_{i,j'}\right)\cup \bigcap_{j'=j-1}^{j-k}\SG\left(Q^{(1)}_{i,j'}\right),
\]
will be postponed to Lemma
\ref{lem:final} below. There $\k$ will be chosen to be large enough but independent of $q$. The requirement
of Lemma \ref{lem:FAperiodic} that $(1-\widehat\nu(\cH)^\k)^{n/(3\k)}< 1/16$ is implied by \eqref{eq:muEi}. 

In the above setting, Lemma
\ref{lem:FAperiodic} gives
\[
 \var_{\L_i}(f\tc \cE_i)\le q^{-O(Rw\l)}\sum_j
  \nu_i\left(\1_{\cH_{i,j}}\var_{Q_{i,j}}(f\tc \cG(Q_{i,j}))\right).\]
By combining the above with \eqref{eq:200} we finally get    
\[
\var_\L(f\tc \cE)\le q^{-O(Rw\l)}\sum_{i,j}
  \mu_{\L}\left(\1_{\cH_{i,j}}\var_{Q_{i,j}}(f\tc \cG(Q_{i,j}))\tc \mkern-12mu\bigcap_{i',j'\in[M]}\mkern-12mu\cG(Q_{i',j'})\right).
\]
We shall now analyse a generic term
in the sum above with the help of Theorem \ref{thm:key step}.
\begin{lem}
\label{lem:final}
There exists an constant $\l=\l(\cU,\d)$ such that the following holds. If the parameter $\k$ of the facilitating events
$\cH_{i,j}$ is taken equal to $\l$ then, for any function $f:\O_\L\to \bbR$ and any $i,j$, 
\begin{multline*}
\mu_{\L}\left(\1_{\cH_{i,j}}\var_{Q_{i,j}}(f\tc \cG(Q_{i,j}))\tc \bigcap_{i',j'}\cG(Q_{i',j'})\right)\\\le
q^{-O(w^4\log^3(1/q)/q^\a)}\mkern-36mu \sum_{\substack{x\in
  \L\\ d(x,Q_{i,j})\le O(\l W)}}\mkern-36mu\mu_{\L}\left(c_x \var_x(f)\right).
  \end{multline*}
\end{lem}
If we assume the lemma, we immediately recover \eqref{eq:9}, concluding the proof of Theorem \ref{thm:periodic}.
\end{proof}

\begin{proof}[Proof of Lemma \ref{lem:final}]
We assume that $\1_{\cH_{i,j}}=1$ and that w.l.o.g.\ the event $\cH_{i,j}^+=\bigcap_{j'=j+1}^{j+\l} \SG\left(Q^{(1)}_{i,j'}\right)$
occurs. Next, we recall Definition \ref{def:snail} of the snail
$V^R_{L}(\ur)$ and we choose $r_l=\rho_{k+l}\lfloor\d
r_{l-1}/\rho_{k+l}\rfloor$ for all $l\in[2k]$, setting $r_{-1}=L=\l (W+2R_0)-2R_0$. We choose $\l$
sufficiently large, depending on $\d$ and $\cU$ but not on $w$  and $q$, in such a way that $Q_{i,j}\subset x+V^R_L(\ur)$, where $x$ is the center of the
rightmost annulus in $Q^{(1)}_{i,j+\l}$. We write $V=x+V^R_L(\ur)$ and observe that, by
construction, 
$Q_{i,j}\cap B=\varnothing$, where $B$ is the base $V$. Finally, we recall Definition \ref{def:good events} of the events
$\SG(B),\cG(T_l^\pm)$ and $\SG(V)=\SG(B)\cap \bigcap_{l\in[2k]}(\cG(T_l^+)\cap \cG(T_l^-))$ for the snail $V$. It is easy to verify the following implications (see Figure \ref{fig:strips}): 
\begin{align}
\label{eq:1A+}\cH^+_{i,j}&{}\subset \cS\cG(B)& \bigcap_{i',j'}\cG(Q_{i',j'})&{}\subset \bigcap_{l\in[2k]}(\cG(T_l^+)\cap \cG(T_l^-)).
\end{align}
Indeed, for the first inclusion, recalling \ref{event:i} it is clear that $\cA$ and $\HA$ occur (since the leftmost annulus in $Q_{i,j+1}$ contains $HA$ and the rightmost one in $Q_{i,j+\l}$ contains $A$) and that all $\SB_{m,p}$ occur (for $SB_{m,p}$ contained in two consecutive squares $Q_{i,j'},Q_{i,j'+1}$ at least in one of them we are guaranteed to have the helping sets; for $SB_{0,p}$ close to the left boundary of $Q_{i,j'}$ the infected rightmost annulus provides the desired helping sets). To see the second one, observe that for all $l,m$, $ST^\pm_{l,m}$ intersects at least one of the squares $Q_{i',j'}$ in a segment of length at least $\varepsilon W$.

Using \eqref{eq:1A+} and $\mu_\L(\cE)=1-o(1)$, we have that
\begin{multline*}
\mu_{\L}\left(\1_{\cH^+_{i,j}}\var_{Q_{i,j}}(f\tc \cG(Q_{i,j}))\tc \bigcap_{i',j'}\cG(Q_{i',j'})\right)\\
\begin{aligned}\le{}& (1+o(1)) \mu_\L\left({\1}_{\SG(B)}\1_{\bigcap_{i',j'}\cG(Q_{i',j'})}\inf_a\mu_{Q_{i,j}}\left(\1_{\cG(Q_{i,j})}(f-a)^2\right)\right)\\
\le{}&(1+o(1)) \mu_\L\left(\inf_a\mu_{Q_{i,j}}\left({\1}_{\SG(V)}(f-a)^2\right)\right)\\
\le{}&\mu_\L\left({\1}_{\SG(V)}\left(f-\mu_V(f\tc\SG(V) )\right)^2\right)/\mu(\SG(V))\\
={}&\mu_\L\left(\var_{V}(f\tc \SG(V))\right).
\end{aligned}
\end{multline*}

If we now apply the bound \eqref{eq:4} of Theorem \ref{thm:key step} and use the fact that $V$ is contained in a deterministic $O(\l W)$-neighborhood of the square $Q_{i,j}$ we get the conclusion of the lemma, once we observe that $c^V_x\le c_x$, where $c_x$ are the constraints on the torus $\L$.
\end{proof}

\section{Open problems}
\label{sec:open pb}
With Theorem~\ref{th:main} establishing universality, the next natural goal is to determine the relaxation time up to a constant factor. This would correspond to reaching the refined universality partition in bootstrap percolation proved in \cite{Bollobas14}. However, for KCM, we expect that the partition, even restricted to the finite stable set case studied in this work, will be more subtle. In order to state it we need one more definition adapted from \cite{Morris17}*{Definition 2.3}.

\begin{defn}
A critical update family of difficulty $\alpha$ is called 
\emph{rooted} if there exist two non-opposite directions of difficulty strictly larger than $\alpha$ and \emph{unrooted} otherwise.
\end{defn}
The importance of this distinction, though not the one initially suggested in \cites{Morris17,Martinelli19a}, is the following. In a rooted model, a droplet may not reproduce in certain directions, which forces it to perform an East-like motion, as explained in Section \ref{sec:sketch}. As it was established in the present work this obstruction can be circumvented using the mechanism explained in Section \ref{sec:sketch} on length scale $q^{-O(1)}$. However, e.g.\ for the rooted model in Figure \ref{fig:example}, typically the first pairs of infections are found at a distance $q^{-2}$ from the droplet and, before reaching them, the droplet can only move right East-like. Therefore, for rooted models with finitely many stable directions we expect that the bottleneck of the dynamics consists in creating $\log (1/q)$ disjoint droplets close to each other. For unrooted models one may hope to directly move droplets in an FA1f-like way (creating a droplet and immediately erasing the previous one), as in Section \ref{sec:Proof}, since they may locally move in all directions.

On the other hand, as identified in \cite{Bollobas14}, balanced models have droplets whose probability of occurrence is $\exp(-\Theta(1)/q^{\alpha})$, while unbalanced ones only have droplets with probability $\exp\left(-\Theta(\log^2(1/q))/q^\a\right)$. Putting these two intuitions together and the mechanism put forward in the present work, this leads us to the following conjecture.
\begin{conj}
\label{conj:logs}
Let $\cU$ be a critical update family with finite number of stable directions and difficulty $\a$. Then
\[\bbE_\mu[\tau_0]=\exp\left(\frac{\log^{\gamma}(1/q)(\log\log(1/q))^{O(1)}}{q^\a}\right),\]
where
\begin{itemize}
\item $\gamma=0$ if there exists at most one direction $u\in S^1$ such that $\alpha(u)>\alpha$ (balanced unrooted models),
\item $\gamma=1$ if there exist at least two directions $u\in S^1$, such that $\a(u)>\a$, but not two opposite ones (balanced rooted models),
\item $\gamma=2$ if there exists $u$ such that $\min(\a(u),\a(u+\pi))>\alpha$, but $\a(v)\le\a$ for $v\neq u,u+\pi$ (unbalanced unrooted models),
\item $\gamma=3$ otherwise, i.e.\ there exist three different directions with difficulty larger than $\a$, two of which are opposite (unbalanced rooted models).
\end{itemize}
Furthermore, we conjecture that the $(\log\log(1/q))^{O(1)}$ correction is in fact $\Theta(1)$, except, possibly, for models with exactly one direction $u\in S^1$ with $\alpha(u)>\alpha$.
\end{conj}
It should be noted that such sharp results are not known for any critical model, so Conjecture \ref{conj:logs} provides the highest precision currently feasible. In fact, the level of precision of the conjecture is not attained for any model, including the most classical FA2f model falling in the first class.\footnote{Since the submission of the present work, a much sharper result was proved for FA2f by the authors \cite{Hartarsky20FA}.}

Since the present work was submitted, in two companion papers {Mar\^ech\'e} and the first author \cite{Hartarsky20I} and the first author \cite{Hartarsky20II}, proved Conjecture \ref{conj:logs}. Moreover, they established that the peculiar $\log\log$ correction left uncertain in the conjecture is indeed present. We refer the reader to those works for more detailed discussions of the bottlenecks and mechanisms involved, particularly for the anomalous case with exactly one direction of difficulty larger than $\alpha$ dubbed \emph{semi-directed} there.

\begin{appendix}
\section*{}
\label{app}
\begin{proof}[Proof of Lemma \ref{lem:FAperiodic}]
We will consider the linear case---the periodic one is treated identically. For simplicity we assume that $2k$ divides $n$. Partition $[n]$ into blocks $I_0,\dots, I_{N-1}$ where 
$I_i:=\{i\k,\dots,(i+1)\k-1\}$ and $N=n/\k$. Let $\cH^{(\ell)}$ be the event that there exists $i$ in the \emph{left half} $[N]^{(\ell)}:=[N/2]$ of $[N]$ such that $\o_j\in \cH$ for all $j\in I_i$. Let $\cH^{(r)}$ be defined similarly but for the blocks with index in the right half $[N]^{(r)}:=[N]\setminus[N]^{(\ell)}$. Using the assumption of the lemma $\nu(\cH^{(\ell)})=\nu(\cH^{(r)})>15/16$ and
\cite{Blondel13}*{Lemma 6.5}, we get
\begin{equation*}
\var_\nu(f\tc \O_\cH)\le 24\nu\left(\1_{\cH^{(r)}}\var^{(\ell)}(f)+\1_{\cH^{(\ell)}}\var^{(r)}(f) \tc \O_\cH\right),
\end{equation*}
where $\var^{(\ell)}$ denotes the variance computed w.r.t.\ the variables corresponding to the blocks in the left half and similarly for $\var^{(r)}$.

Given $\cH^{(\ell)}$, let $\xi$ be the smallest label in $[N]^{(\ell)}$ such that $\o_j\in \cH$ for all $j\in I_{\xi}$. Using Lemma~\ref{lem:var:conv} and the fact that the event $\{\xi=i\}$ is independent of the variables $(\o_j)_{j\ge (i+1)}\k$, we get that
\begin{equation}
\begin{aligned}
\label{eq:A1}
  \nu\left(\1_{\cH^{(\ell)}}\var^{(r)}(f) \tc \O_\cH\right)&{}\le \sum_{i\in [N]^{(\ell)}}\nu\left(\1_{\{\xi=i\}}\var_{\ge (i+1)\k}(f)\tc \O_\cH\right)\\
  &{}\le\frac{1}{\widehat\nu(\cH)^{\k}}\sum_{i\in [N]^{(\ell)}}\nu\left(\1_{\{\xi=i\}}\var_{\ge (i+1)\k}(f)\right),
\end{aligned}
\end{equation}
where $\var_{\ge (i+1)\k}(f)$ is the variance w.r.t.\ the variables $(\o_j)_{j\ge (i+1)\k}$. The r.h.s.\ above can now be bounded above using \cite{Martinelli19a}*{Proposition 3.4}. If $\cH_{I_j}$ is the event that $\o_l\in \cH$ for all $l \in I_{j}$, with the convention that $\cH_{I_N}$ and $\cH_{I_{-1}}$ do not occur, 
we get that 
\begin{equation*}
  \1_{\{\xi=i\}}\var_{j\ge (i+1)\k}(f)\le \frac{1}{\widehat\nu(\cH)^{O(\k)}}\sum_{j=i+1}^{N-1}\nu_{\ge (i+1)}\k\left(\1_{\{\xi=i\}}\1_{\cH_{j}^\pm}\var_{I_j}(f)\right), 
\end{equation*}
where $\cH_{j}^\pm=\cH_{I_{j-1}}\cup\cH_{I_{j+1}}$ and $\var_{I_j}$ is the variance w.r.t.\ the variables in $I_j$. By inserting the r.h.s.\ above into the r.h.s.\ of \eqref{eq:A1}, we obtain that $\nu\left(\1_{\cH^{(\ell)}}\var^{(r)}(f) \tc \O_\cH\right)$ is smaller than
\begin{multline*}
\frac{1}{\widehat\nu(\cH)^{O(\k)}}\nu\left(\sum_{j=1}^{N-1}\sum_{i=0}^{j-2}\1_{\{\xi=i\}}\1_{\cH_j^\pm}\var_{I_j}(f)+\sum_{i\in[N]^{\ell}}\1_{\{\xi=i\}}\1_{\cH^\pm_{i+1}}\var_{I_{i+1}}(f)\right)\\
\le \frac{2}{\widehat\nu(\cH)^{O(\k)}}\sum_{j=1}^{N-1}\nu\left(\1_{\cH_j^\pm} \var_{I_j}(f)\right),
\end{multline*}
where we have isolated the term $j=i+1$ and used $\sum_i\1_{\{\xi=i\}}\le 1$ and $\1_{\{\xi=i\}}\le 1$ for the two terms respectively. Exactly the same argument can be applied to the term $\nu\left(\1_{\cH^{(r)}}\var^{(\ell)}(f)\tc\O_\cH\right)$ to conclude that
\begin{equation}
\label{eq:A:varnu}
\var_\nu(f\tc\O_\cH)\le  \frac{96}{\widehat\nu(\cH)^{O(\k)}}\sum_{j=0}^{N-1}\nu\left(\left(\1_{\cH_{I_{j+1}}}+\1_{\cH_{I_{j-1}}}\right)
\var_{I_j}(f)\right).
\end{equation}
We finally bound from above a generic term, 
considering $\nu\left(\1_{\cH_{I_1}}
\var_{I_0}(f)\right)$ for concreteness.

We apply Lemma
\ref{lem:2block} with $X_1=\o_{\k-1}, X_2=(\o_0,\o_1,\dots,\o_{\k-2})$ and facilitating event $\{\o_{\k-1}\in \cH\}$ to $\var_{I_0}(f)$ in order to get 
\begin{equation}
  \label{eq:A4}
 \var_{I_0}(f)\le \frac{2}{\widehat\nu(\cH)}\nu_{I_0}\left(\var_{\k-1}(f)
 +\1_{\{\o_{\k-1}\in \cH\}}\var_{I_0\setminus \{\k-1\}}(f)\right).
\end{equation}
Thus, we obtain

\begin{align*}
\nu\left(\1_{\cH_{I_1}}
\var_{I_0}(f)\right)\le{}&
\frac{2}{\widehat\nu(\cH)}\left(\nu\left(\1_{\cH_{I_1}} \var_{\k-1}(f)\right)
+\nu\left(\1_{\cH_{I_1}} \1_{\{\o_{\k-1}\in \cH\}}\var_{I_0\setminus\{\k-1\}}(f)\right)\right)\\ \le{}&\frac{2}{\widehat\nu(\cH)}\left(\nu\left(\1_{\cH_{\k-1}}\var_{\k-1}(f)\right)
+\nu\left(\1_{\cH_{\k-2}}\var_{I_0\setminus\{\k-1\}}(f)\right)\right)
\end{align*}
We can repeat the step leading to \eqref{eq:A4} with $X_1=\o_{\k-2}, X_2=(\o_0,\dots,\o_{\k-3})$ and facilitating event $\{\o_{\k-2}\in \cH\}$ and so on.
At the end of the iteration we finally get that
\[
\nu\left(\1_{\cH_{I_1}}\var_{I_0}(f) \right)\le \left(\frac{2}{\widehat
  \nu(\cH)}\right)^{\k}\sum_{i\in I_0}\nu\left(\1_{\cH_i}\var_{i}(f)\right).
\]
Putting all together, we have finally proved that
\begin{align}
\var_\nu(f)&\le \frac{96}{\widehat \nu(\cH)^{O(\k)}}\sum_{j\in[N]}\nu\left(\left(\1_{\cH_{I_{j+1}}}+\1_{\cH_{I_{j-1}}}\right)
\var_{I_j}(f)\right)\nonumber \\
&\le  \left(\frac{2}{\widehat\nu(\cH)}\right)^{O(\k)}\sum_{j}\sum_{i\in I_j}\nu\left(\1_{\cH_i}\var_{i}(f)\right).\tag*{\qedhere}
\end{align}
\end{proof}
\end{appendix}

\section*{Acknowledgments}
We wish to thank Laure Mar{\^e}ch\'e for many enlightening discussions concerning universality for $\cU$-KCM.

\bibliographystyle{imsart-number}
\bibliography{Bib}
\end{document}